\documentclass{amsart}
\usepackage{amsfonts}
\usepackage{amsmath,amscd}
\usepackage{amssymb}
\usepackage{amsthm}
\usepackage{newlfont}
\usepackage{graphicx}
\usepackage[all]{xy}
\usepackage[usenames]{color}
\usepackage{comment}
\usepackage{breqn}
\usepackage{cancel}
\usepackage{amsthm,amsmath,amssymb,hyperref,amsfonts,tikz,adjustbox}

\usepackage[normalem]{ulem}

\newcommand{\E}{\mathcal{E}}
\newcommand{\R}{\mathbb{R}}
\newcommand{\C}{\mathbb{C}}
\newcommand{\F}{\mathbb{F}}

\newcommand{\D}{{\mathcal D}}
\newcommand{\0}{{\vartheta}}
\newcommand{\X}{{\mathcal X}}
\newcommand{\m}{{}^{-1}}

\makeatletter
\def\moverlay{\mathpalette\mov@rlay}
\def\mov@rlay#1#2{\leavevmode\vtop{%
   \baselineskip\z@skip \lineskiplimit-\maxdimen
   \ialign{\hfil$\m@th#1##$\hfil\cr#2\crcr}}}
\newcommand{\charfusion}[3][\mathord]{
    #1{\ifx#1\mathop\vphantom{#2}\fi
        \mathpalette\mov@rlay{#2\cr#3}
      }
    \ifx#1\mathop\expandafter\displaylimits\fi}
\makeatother

\long\def\alert#1{\parindent2em\smallskip\hbox to\hsize
{\hskip\parindent\vrule
\vbox{\advance\hsize-2\parindent\hrule\smallskip\parindent.4\parindent
\narrower\noindent#1\smallskip\hrule}\vrule\hfill}\smallskip\parindent0pt}
 \newtheorem{thm}{Theorem}[section]
\newtheorem{cor}[thm]{Corollary}
 \newtheorem{lem}[thm]{Lemma}
 \newtheorem{prop}[thm]{Proposition}
\theoremstyle{definition}
 \newtheorem{defn}[thm]{Definition}
\theoremstyle{remark}
 \newtheorem{rem}[thm]{Remark}
  
 \newtheorem{ex}[thm]{Example}

 \numberwithin{equation}{section}
\newtheorem*{Theorem A}{\textbf{Theorem A}}
\newtheorem*{Corollary E}{\textbf{Corollary E}}
\newtheorem*{Corollary F}{\textbf{Corollary F}}
\newtheorem*{Theorem B}{\textbf{Theorem B}}
\newtheorem*{Theorem C}{\textbf{Theorem C}}
\newtheorem*{Theorem D}{\textbf{Theorem B}}
\newtheorem*{p a}{\textbf{Proof of  Theorem A}}
\newtheorem*{p b}{\textbf{Proof of Theorem B}}
\newtheorem*{p c}{\textbf{Proof of  Theorem C}}
\newtheorem*{p d}{\textbf{Proof of  Theorem D}}
\DeclareMathOperator{\der}{Der}
\DeclareMathOperator{\dom}{dom}
\DeclareMathOperator{\im}{im}
\DeclareMathOperator{\en}{End}
\DeclareMathOperator{\pen}{\textsf{PEnd}}
\DeclareMathOperator{\pde}{\textsf{PDer}}
\DeclareMathOperator{\ad}{ad}
\DeclareMathOperator{\Hom}{Hom}

\begin{document}

%
\title[Inverse semialgebras and Partial actions  of Lie algebras]{Inverse semialgebras and Partial actions  of Lie algebras}
\author{Mikhailo Dokuchaev}
\address{Departamento de Matem\'atica, Universidade de S\~ao Paulo,
Rua do Mat\~ao, 1010, 05508-090 S\~ao Paulo, Brazil.}
\email{dokucha@ime.usp.br}

\author{Farangis Johari}
\address{Universidade Federal do ABC, Center for Mathematics, Computing and Cognition, Av. dos Estados, 5001-Bangú, Santo André-SP, 09280-560, Brazil.}
\email{farangis.johari@ufabc.edu.br}

\author{ Jos\'e L. Vilca-Rodríguez}
\address{Departamento de Matem\'atica, Universidade de S\~ao Paulo,
Rua do Mat\~ao, 1010, 05508-090 S\~ao Paulo, Brazil.}
\email{jvilca@ime.usp.br}


\thanks{\textit{Mathematics Subject Classification 2020.} Primary: 20M18 Secondary: 17B99, 16W25, 06A12, 17A99.}

\keywords{Inverse semigroup, Lie algebra, inverse semialgebra, partial derivation, partial action, premorphism, partial representation.}

\date{\today}

\begin{abstract} We introduce the concept of a non-associative (i.e. non-necessarily associtive) inverse semialgebra over a field,
the Lie version of which is inspired by the set of all partially defined derivations of a non-associative  algebra, whereas the associative case is based  on such examples as the set of all partially defined linear maps of a vector space, the   set  of all sections of the structural sheaf  of a scheme,  the  set  of all regular functions  defined on open subsets of an algebraic variety and   the set of all smooth real-valued functions defined on open subsets of a smooth manifold. Given a Lie algebra $L$ we define the notion of a partial action of $L$ on a non-associative algebra $A$ as an appropriate premorphism and introduce a  Lie inverse semialgebra $E(L),$ which is a Lie analogue of R. Exel's inverse semigroup $S(G)$ that governs the partial actions of a group $G.$ We  discuss how $E(L)$ controls the premorphisms from $L$ to $A,$ obtaining results on its total control. We define the concept of an $F$-inverse Lie semialgebra and obtain Lie theoretic analogues of some classical results of the theory of inverse semigroups, namely, we show that the category of partial representations of $L$ in  meet semilattices is equivalent to the category   ${\mathcal F}$ of $F$-inverse Lie semialgebras with morphisms that preserve the greatest elements of $\sigma$-classes. In addition,   we establish an adjunction between the category of Lie algebras and the category ${\mathcal F}.$

 \end{abstract}
\maketitle

\tableofcontents


\section{Introduction}

 It is well known that  the  partial symmetries of a  set $X$ form the notorious  symmetric inverse semigroup ${\mathcal I}(X)$  under the composition of the partial bijections taken on the largest possible domain (see \cite{Lawson}). Moreover,  partial group actions on  a set are closely related to inverse semigroups, in particular, to ${\mathcal I}(X).$   Indeed, on the one hand,  Exel's definition of a partial action of a group $G$ on $X,$ given in  \cite[Definition 1.2]{Exel1998}, can be seen as a
  map $G\to {\mathcal I}(X),$ satisfying properties which characterize what we call a partial represention of $G$ into an inverse semigroup \cite[Proposition 4.1]{Exel1998}. In addition, a partial action of $G$ on $X$ can also be defined as a  premorphism of the form $G \to {\mathcal I}(X)$ (see  \cite[Proposition 2.1]{KL}).
  On the other hand, in the same paper \cite{Exel1998}  R. Exel associated to any group $G$ an inverse semigroup $S(G),$  defined by generators and relations, such that the actions of $S(G)$ on a set $X$ are  in a one-to-one correspondence  with the partial actions of $G$ on $X.$ Later,  J. Kellendonk and   M. V. Lawson \cite{KL} proved that $S(G)$  is isomorphic to the Szendrei expansion of $G$ \cite{Szendrei},  the latter being also isomorphic to the Birget-Rhodes prefix expansion of $G$  \cite{BR2}.

The above mentioned  Exel's definition was given as an abstraction of the notion of a partial group action    on a $C^*$-algebra, worked out in the theory of operator algebras in order to  describe important classes of $C^*$-algebras as more general crossed products (see~\cite{exel1994,Mc,E0}). This approach has yielded remarkable results concerning  representations, ideal structures, and $K$-theory of algebras under consideration. A closely related concept (also referred to as a ``partial action") can be found in~\cite{GreenMarcos}, where the authors explore graded modules over quotients of path algebras. 

J.~Kellendonk and M.~Lawson~\cite{Lawson} highlighted the  importance of partial actions for $\mathbb{R}$-trees, model theory, the profinite topology of groups, Fuchsian groups, tilings of Euclidean space, graph immersions and inverse semigroups, as well as topology and group presentations.  Subsequent developments  have extended the theory  to include partial actions of semigroups~\cite{Megre2,Hol1,GouHol1,CornGould,Kud,
Khry1,KudLaan2}, groupoids~\cite{Gilbert,BP,CaenFier,NysOinPin,MaPi}, and categories~\cite{Nystedt}, as well as to the setting of partial   (co)actions of Hopf algebras (or weak Hopf algebras)~\cite{CaenJan,AB3,ABV,CasPaqQuaSant,
BatiVerc1,HuVerc,AzMaPaSi,BaHaSaVe}.
In algebra the new concepts are   useful   to graded algebras,   Hecke algebras,   Leavitt path algebras,  
  inverse semigroups,  restriction  semigroups  and automata (see the survey article \cite{dokuchaev2019} and the references therein). 
  
  More  recent applications include the use of partial actions    to dynamical systems associated to separated graphs and related $C^*$-algebras   \cite{AraE1}, \cite{AraL},  to paradoxical decompositions \cite{AraE1}, to shifts  \cite{AraL}, to full or reduced $C^*$-algebras  of $E$-unitary or strongly $E^*$-unitary inverse semigroups  \cite{MiSt}, 
to  topological higher rank graphs \cite{RenWil}, to Matsumoto and Carlsen-Matsumoto $C^*$-algebras of arbitrary subshifts    \cite{DE2},   to ultragraph $C^*$-algebras \cite{GR3} and to universal $F$-inverse monoids \cite{KudFur}.
 For more information around partial actions and their
applications we refer the reader to R. Exel’s book~\cite{ExelBook} and the surveys~\cite{Ba,dokuchaev2011,dokuchaev2019}.

In view of the numerous interesting developments around partial group actions and applications, it is natural to ask whether a similar approach can be applied to actions by derivations on algebras. The main motivation of this paper is to establish a suitable framework in which an analogue of the theory of partial group actions can be developed in the setting of Lie algebras  and bring into consideration new interesting algebraic structures which naturally appear this way.

It is well known that Lie algebras act  on  algebras by derivations. This observation leads to the following  question:

\begin{center}
\emph{How can a Lie algebra act partially on a  non-necessarily associative algebra?}\\
\end{center}

 Since partial groups actions can be obtained by restricting global actions, it is reasonable to  approach this question looking  at restrictions of global (usual) actions of Lie algebras. Specifically, given an action of a Lie algebra $L$ over a field $\F$ on a  non-associative (i.e. non-necessarily associative)  $\F$-algebra $B$, and a subalgebra $A$ of $B$, we consider a natural restriction of this action to $A$. This way  we obtain the so-called {\em ``partial derivations"} of $A$, which are linear maps defined on subalgebras of $A$ that satisfy the Leibniz identity. It is worth noting that partial derivations (particularity those defined on ideals) appear in the construction of the algebra of quotients of a Lie algebra. The latter was introduced  by M. S. Molina in~\cite{Molina2004}, followed by   interesting developments and applications (see, for example,~\cite{BresarPereraOrtegaMolina} and \cite{OrtegaMolina}).

In order to obtain an analogue of the inverse semigroup of all partial symmetries in our context, we consider the set $\pde(A)$ of all partial derivations of the non-associative algebra $A$. This set can be equipped with an addition $+:\pde(A)\times \pde(A)\to \pde(A)$ and a product (Lie bracket) $[\cdot,\cdot]:\pde(A)\times \pde(A)\to \pde(A)$, each operation being defined on the largest subalgebra of $A$ where it makes sense (see Section~\ref{s: preliminaries}), as well a  scalar multiplication $\F\times \pde(A)\to \pde(A)$, which is compatible with the previous operations. By axiomatizing the properties of this structure, we arrive at the  notion of what we call a {\em ``Lie inverse semialgebras''} (see Definition~\ref{def6}). In particular, $(\pde(A),+)$ is a commutative inverse semigroup. When combined with scalar multiplication, this gives rise to a relaxed version of the notion of a vector space, where the underlying abelian (additive) group is replaced by a commutative (additive) inverse semigroup. We refer to this structure as an {\em ``inverse semivector space''} (see Definition~\ref{def2}).

On the other hand, if we focus exclusively on the vector space structure of $A$, then each partial derivation can be regarded simply as a linear map from a subspace of $A$ into $A$. In this setting, it is natural to consider the set $\pen(V)$ of all partial (linear) endomorphisms of a vector space $V.$ Equipped with addition and scalar multiplication as previously described, $\pen(V)$ forms an inverse semivector space. Furthermore, we can define a {\em ``partial composition''} of two partial endomorphisms as their composition on the largest subspace where it is well-defined. Axiomatizing the properties of this structure leads to the notion of an {\em ``associative inverse semialgebra''.} This structure arises naturally in many important contexts,  including  $\pen(V).$ Examples of associative inverse semialgebras (see Example~\ref{example:Sheaves} for details)
include:
\begin{enumerate}
\item  All sections of the structural sheaf of a scheme  over  a field.
\item All sections of the sheaf of regular functions of an algebraic variety.
\item The set of all locally defined smooth real-valued functions on a smooth manifold.
\item All sections of a sheaf of $C^*$-algebras of locally defined bounded continuous complex-valued functions on a topological space.
\end{enumerate}
 This work is devoted to develop the foundations of the notions mentioned above, including that  of  a \emph{partial action of a Lie algebra} on a non-associative algebra (see Definition~\ref{def:partialAc}), and to study the connection between them.

The structure of the paper is as follows. Section~\ref{s: preliminaries} addresses the preliminary notions. In particular, we introduce the concept of an \emph{inverse semivector space} over a field \( \F \). As mentioned earlier, this is a relaxed version of the notion of a vector space. Also, it is important to note that this concept is also a weaker version of the  known notion of a semimodule over a ring with a unit (in particular, over a field), as explained before Definition~\ref{def2}. We also  define the notion of a \emph{non-associative inverse semialgebra}, as well as the concept of a \emph{Lie inverse semialgebra} and that of an \emph{associative inverse semialgebra}. We provide several  examples  and establish some properties needed  for subsequent sections. Given a Lie algebra \( L \), we introduce the Lie inverse semialgebra \( E(L) \) (see Proposition~\ref{p: E(L)}), which is pivotal in our study.  It is inspired by the above mentioned Exel's inverse semigroup \( S(G) \),   which governs  the partial actions of a group  \( G \) on sets and whose algebra rules the partial representations of  \( G \) in algebras.

In Section~\ref{s: semillatices}, we study a special class of non-associative inverse semialgebras, namely, {\em semilattices of non-associative algebras}. Roughly speaking, these are non-associative inverse semialgebras in which the (additive) idempotents have well-behaved interactions with the product (see Definition~\ref{eq:semilatticeAlg}). More specifically, we focus on semilattices of  associative (or Lie) algebras. An important feature of this structure is its close relationship with presheaves of   associative (or Lie)  algebras, as  we  show in Proposition~\ref{prop-semila}. Furthermore, Corollary~\ref{c: semillatice} establishes that any semilattice of Lie algebras can be embedded into a semilattice of associative algebras.

In Section~\ref{s: partial actions}, we introduce the notion of a premorphism between Lie inverse semialgebras. Inspired by the concept of partial group actions, we define a partial action of a Lie algebra $L$ on a non-associative algebra $A$ as a premorphism from $L$ to $\pde(A)$. One of the main results in  the section is Theorem~\ref{prouniq}, which essentially states that the  above mentioned Lie inverse semialgebra $E(L)$ governs the premorphisms defined on $L$ in a specific way. This result serves as a Lie-theoretic analogue of a well-known result by R. Exel~\cite[Proposition 2.2]{Exel1998} (see also~\cite[Theorem 2.4]{KL}).  Moreover,  Corollary~\ref{cor:tilderho} says that  $E(L)$ controls perfectly the premorphisms from $L$ to  semilattices of algebras.   Another important consequence is  Corollary~\ref{cor:classes} , which  establishes a one-to-one correspondence between  partial actions of $L$ and actions of $E(L)$ when dealing with some important specific classes of algebras.

 An interesting class of inverse semigroups close to groups, which keeps drawing attention  of researchers (see, for instance, \cite{KudFur}) is formed by  $F$-inverse semigroups, and it is natural to consider their Lie  analogue in order to look at Lie inverse   semialgebras, which are close to Lie algebras. It is well known that the category of the so-called $F$-pairs (following the terminology of M.~Petrich in~\cite{Petrich1984}) and their morphisms is equivalent to the category of $F$-inverse semigroups with morphisms that preserve the greatest elements of $\sigma$-classes~\cite[VII.6]{Petrich1984}. An important result due to  M.~Szendrei \cite{Szendrei}  states that there exists an adjunction between  the latter category and the category of groups. In Lemma~\ref{l: f-pairs}, we observe that an $F$-pair can be identified with what is now commonly referred to as a unital partial representation of a group in a meet semilattice,  which can  be also equivalently seen as a unital partial action of the group on the semilattice. 

In Section~\ref{s: F-inverse}, we introduce the notion of a {\em partial representation} of a Lie algebra $L$ in a unital meet semilattice~$\Lambda$, defined as a partial representation of the additive group $(L,+)$ in $\Lambda$ that satisfies an additional compatibility condition. We also adapt the notions of 
 a congruence and the minimum group congruence from semigroup theory to our setting. Subsequently, we present the concept of an {\em $F$-inverse Lie semialgebra} and establish several  basic results concerning partial  representations and $F$-inverse Lie semialgebras. A central  fact of the section is Theorem~\ref{pro-equivalence}, which establishes a categorical equivalence between partial representations of $L$ in unital meet semilattices and $F$-inverse Lie semialgebras with morphisms that preserve the greatest elements of $\sigma$-classes. This can be regarded as a Lie-theoretic analogue of the categorical equivalence from~\cite[VII.6]{Petrich1984}, as  mentioned in the previous paragraph. Furthermore, in Theorem~\ref{theorem-action}, we prove the existence of an adjunction between the category of Lie algebras and the category of $F$-inverse Lie semialgebras with morphisms preserving the greatest elements of $\sigma$-classes. This result serves as a Lie-theoretic counterpart to the adjunction for groups and $F$-inverse semigroups established by M.~Szendrei.

\section{Notions and Preliminaries: Lie and associative inverse semialgebras}\label{s: preliminaries}

In this paper, given  a commutative additive inverse semigroup $S$ we denote by $\E(S)$ the set of all idempotents of $S$. Furthermore, if $S$ is a monoid then $0$ denotes the identity element of $S$. For convenience, we use the notation  $0_x:=x +(-x)$ for any $x\in S.$ The sum $
x +(-x)$ will be simply written as $x-x.$ It follows immediately that $\E(S)=\{0_x\mid x\in S\}$, $0_{x+y}=0_x+0_y$, and $0_e=e$ for any $e\in \E(S)$. We recall that in an inverse semigroup $S$, we have the following natural partial order: 
\[x\preceq y \Longleftrightarrow x=y+e~~\text{for some } e\in \E(S). \] It can be observed that 
 \begin{equation}\label{eq_order}
    x\preceq y \Longleftrightarrow x=y+0_{x}\Longleftrightarrow 0_x=y-x.
\end{equation}

Basics on inverse semigroups can be found, for example, in \cite{Lawson} and  \cite{Petrich1984}.

It is well-known that the  partial symmetries of a given set form an inverse semigroup under the composition of the partial bijections (see \cite{Lawson}). Moreover, the study of partial group actions on a set is closely related to the semigroup of all partial symmetries of that set, as established by R. Exel in 
\cite{Exel1998} (see also \cite{KL}). A natural question arises: can a similar approach be applied to actions by derivations on algebras?  Dealing with this question requires us to consider  what we call an ``{\em inverse semivector space}'', which is a relaxed version of the notion of a vector space, in which the  abelian (additive) group is replaced by a commutative (additive) inverse semigroup. In order to compare this idea with a known concept of a semimodule over a semiring we recall the formal definition of the latter.

\begin{defn}[see~\cite{Golan,Zimmermann}]\label{defsem}
Let $(S,\cdot,+)$ be a semiring.  An additive commutative semigroup $M$ is called a {\em semimodule} over $S$ if  for all $x,y\in M$ and $\alpha,\beta\in S$ the following conditions hold:
\begin{enumerate}
      \item $(\alpha+\beta)x=\alpha x+\beta x$; 
      \item $\alpha (x+y)=\alpha x+\alpha y$;
      \item $\alpha (\beta x)=(\alpha \beta)x$.
\end{enumerate}
Additionally, if $(S,\cdot)$ is a monoid, we also require:
\begin{enumerate}
    \item[(4)] $1x=x$;
\end{enumerate}
and if $(S,+)$ and $M$ are monoids, we further require: 
\begin{enumerate}
    \item[(5)] $0x=0$.
\end{enumerate}
\end{defn}

It is worth noting that if in the  above definition  $S$ is a ring with unity (in particular,  a field) and $M$ is a monoid, then conditions~(1), (4) and (5) imply that $M$ must be a group: the inverse of $x$ will be  $(-1)x$. As a result, we recover the usual notion of a module over a ring. This tells us that considering semimodules over rings does not yield anything new. This poses a problem if we aim to consider notions of linear algebra for commutative inverse monoids over fields, which is crucial for our purposes. The issue arises because the multiplication of any element by the zero scalar always collapses to the zero element in the inverse monoid, causing a loss of the structural information provided by the other idempotents. One way to address this problem is to exclude axiom (5) from Definition~\ref{defsem} for inverse monoids over fields (or rings with unity). We adopt this convention and highlight it in the following definition.  The choice of the term  ``{\em inverse semivector space}'' seems to be very natural, and, moreover, it seems to avoid confusion with existing concepts in the literature.

\begin{defn}\label{def2}
An {\em inverse semivector space} over a field $\F$  is a commutative additive inverse semigroup $(V,+)$ endowed with a multiplication by scalars $\mathbb{F} \times V \rightarrow V$, denoted by $(\alpha,x)\mapsto \alpha x$, which for every $x,y\in V$ and $\alpha,\beta\in \F$ satisfies:
\begin{itemize}
       \item[$(1)$]$(\alpha+\beta)x=\alpha x+\beta x$; 
        \item[$(2)$]$\alpha (x+y)=\alpha x+\alpha y$;
         \item[$(3)$]$\alpha (\beta x)=(\alpha \beta)x$;
          \item[$(4)$]$1x=x.$
   \end{itemize}
\end{defn}


\begin{lem}\label{p1f}
Let $(V,+)$ be an inverse semivector space over the field $\mathbb{F}$. Then, for any $x\in V$ and $\alpha \in \mathbb F$ we have that:
\begin{enumerate}
\item $(-\alpha )x=\alpha (-x)=-\alpha x;$ 
\item $0x=0_{\alpha x}.$ In particular, $0_x=0_{\alpha x}.$ 
\item $\alpha 0_x = 0_x,$ that is $\alpha e=e$ for any $e\in \E(V).$
\end{enumerate}
    \end{lem}
\begin{proof}
  (1) The  property $(-\alpha) x=-\alpha x$ follows from the equalities $$\alpha x=(\alpha -\alpha+\alpha)x=\alpha x+(-\alpha)x+ \alpha x \quad \mbox{and} \quad (-\alpha) x=(-\alpha +\alpha-\alpha)x=(-\alpha) x+\alpha x+ (-\alpha) x,$$ as $(V,+)$ is an inverse semigroup. The other equality is proven analogously.

(2) By item (1) we have $0x=(\alpha-\alpha)x=\alpha x+(-\alpha)x=\alpha x-\alpha x=0_{\alpha x}.$ The remaining assertion follows immediately.

(3) Keeping in mind item (2) we see that $\alpha 0_x = \alpha (x-x) = \alpha x - \alpha x = 0_{\alpha x} = 0_{x}.$
\end{proof}

 Note that by (4) of Definition~\ref{def2} and (1) of Lemma~\ref{p1f} we have that 
$$(-1)\cdot x   = -(1 \cdot x) = -x$$ 
for all $x$ in an inverse semivector space $V.$ 

An {\em inverse semivector subspace} of $V$ is subsemigroup   $W$ of $(V,+),$  which is closed under the multiplication by the scalars. If $V$ and $W$ are inverse semivector spaces, then a mapping
 $\phi: V\to W$ is called {\em linear}  if for all $x,y\in V$ and $\alpha\in \F$, we have $\phi(x+y)=\phi(x)+\phi(y)$ and $\phi(\alpha x)=\alpha\phi(x)$. Additionally, if $(V,+)$ and $(W,+)$ are monoids, we also require that $\phi(0)=0$.

\begin{rem}\label{rem:linear} Observe that for the linearity of a mapping  $\phi: V\to W$ between  inverse semivector spaces it is enough to require that 
 $$\phi(x+y)=\phi(x)+\phi(y)\;\;\; \text{and}\;\;\; \phi(\alpha x)=\alpha\phi(x)$$  for all $x,y\in V$ and $0\neq \alpha\in \F .$ Indeed, for $\alpha =0$ we see, using Lemma~\ref{p1f},  that 
 $$\phi (0 x) = \phi (0_x) = \phi (x+(-1)x) =
 \phi (x) + \phi((-1)x)= \phi (x) + (-1) \phi(x)= 
 \phi (x) - \phi(x)= 0_{\phi(x)}= 0 \phi(x).$$
 \end{rem}

\subsection{Inverse semialgebras}

Suppose that  an inverse semivector space $S$ over  the field $\mathbb{F}$ is endowed with a map $\cdot:S \times S \rightarrow S,$ wich we shall denote by $(x,y)\mapsto x\cdot y$ and call {\it a multiplication map}. We shall say that $S$   is {\em right distributive} if 
\begin{equation}\label{RightDistr}
     (y+z)\cdot x=y\cdot x+z\cdot x 
\end{equation} 
 for all  $x,y,z\in S.$ Symmetrically we define the left distributive property, and we shall say that $S$   is {\em distributive} if it is both left and right distributive.

\begin{lem}\label{lemma_distributive} Let $S$ be an inverse semivector space $S$ over  the field $\mathbb{F}$  endowed with a  multiplication map.
\begin{enumerate} 
\item If $S$ is  right distributive then then $0_x \cdot y=0_{x\cdot y},$ for all $x,y\in S.$ 
\item If $S$ is  left distributive then then $x \cdot 0_y=0_{x\cdot y},$ for all $x,y\in S.$ 
\item If $S$ is  distributive  then $0_x \cdot 0_y=0_{x\cdot y},$ for all $x,y\in S.$ 
 \end{enumerate}
\end{lem}
\begin{proof}
(1) We see that $0_x\cdot y=(x-x)\cdot y=x\cdot y-x\cdot y=0_{x\cdot y}$ and (2) follows symmetrically.
(3) Using (1) and (2) we obtain $0_x \cdot 0_y = 0_{x \cdot 0_y}= 0_{0_{x\cdot y}}   =  0_{x \cdot y},$ as desired{\color{red} .}
\end{proof}

  Usually  ``{\em non-associative}'' means ``{\em non necessarily associative}''. Keeping this in mind we give the next:

\begin{defn}\label{def4}
A  {\em non-associative inverse semialgebra}   over the field $\mathbb{F}$ is an inverse semivector space $S$ over  $\mathbb{F}$ together with 
a map  $\cdot:S \times S \rightarrow S,$ denoted by $(x,y)\mapsto x\cdot y,$ such that for every $x,y,z\in S$ and $0\neq \alpha\in \F$ we have:
\begin{enumerate}
\item $\alpha (x\cdot y)=(\alpha x)\cdot y=x\cdot (\alpha y);$

\item $x\cdot ( y+z)\succeq x\cdot y+x \cdot z;$
\item $(x+z)\cdot y\succeq x\cdot y+z\cdot y;$

\item $x\cdot ( e+z)= x\cdot e+x \cdot z$ for all $e\in \E(S);$
 \item $(x+e)\cdot z= x\cdot z+e\cdot z$  for all $e\in \E(S);$
 \item $e+f \preceq  e\cdot f \preceq  f $  for all $e,f\in \E(S).$
\end{enumerate}
\end{defn}

We shall say that  a non-associative inverse semialgebra  $S$ over the field  $\mathbb{F}$ is {\em associative}  if  $$(x\cdot y)\cdot z=x\cdot ( y\cdot z)$$
for all $ x,y,z \in S.$

\begin{lem}\label{lemma_basicproperties}
Let $S$ be a  non-associative inverse semialgebra   over the field $\mathbb{F}.$ Then

\begin{enumerate}
\item $e \cdot f + f\cdot e = e+f$ for all $e,f \in \E (S);$
\item $x\cdot 0_y, \, 0_x \cdot y \in \E(S)$ for all $x,y\in S;$
     \item 	$0_{x\cdot y}\preceq 0_x\cdot y,$ 
     $0_{x\cdot y}\preceq x \cdot 0_y$ and $0_{x\cdot y}\preceq 0_x \cdot 0_y$ for all $x,y\in S;$ 
    \item $x \cdot (y+z)-(x\cdot y+x\cdot z) = 0_{x\cdot y}+0_{x\cdot z}$ for all $x,y,z\in S;$
\end{enumerate} 
\end{lem}
\begin{proof}
(1)  By item (6) of Definition~\ref{def4}  we have  $e+f \preceq  e\cdot f \preceq  f $ and  $f+e \preceq  f\cdot e \preceq  e ,$ for any $e,f \in \E (S),$ which gives  $$e+f \preceq  e\cdot f + f \cdot e \preceq  e+ f,  $$ implying the desired equality.

(2)  By item (4) of Definition~\ref{def4}   we obtain that
$$ x\cdot 0_y + x\cdot 0_y = x\cdot (0_y +  0_y)
=  x\cdot 0_y $$ and, analogously, $ 0_x \cdot y \in \E(S).$ 

(3) By (4) and  (5)  of Definition~\ref{def4} we see that 
$$0_{x\cdot y}=x\cdot y-x\cdot y=(x+0_x)\cdot (y+0_y)-x\cdot y=x\cdot y + x\cdot 0_y + 0_{x}\cdot y+0_{x}\cdot 0_y -x\cdot y = 0_{x\cdot y} + x\cdot 0_y + 0_{x}\cdot y+0_{x}\cdot 0_y,$$ which implies the desired inequalitites in view of item (2).
 
(4) Since,  $x\cdot y+x\cdot z\preceq x\cdot (y+z)$ (see (3) of Definition~\ref{def4}),  we obtain that
$$x\cdot (y+z)-(x\cdot y+x\cdot z)=0_{x\cdot y+x\cdot z}=0_{x\cdot y}+0_{x\cdot z}.$$
\end{proof}

Notice that by (2) of 
 Lemma~\ref{lemma_basicproperties}  we have that $ef \in \E(S)$  for any $e,f \in \E (S).$ In addition, keeping in mind (2) of Lemma~\ref{p1f}, observe that  
 (3) of  Lemma~\ref{lemma_basicproperties} says  that 
     $$0 (x\cdot y) \preceq (0x)\cdot y, \;\;\; 0(x\cdot y)\preceq x \cdot (0y)  \;\;\; \text{and} \;\;\; 
     0(x\cdot y)\preceq (0x) \cdot (0y)$$ for all $x,y\in S.$

Let $V$ be a vector space over $\mathbb{F}.$ A {\em partial (linear) endomorphism} of $V$ is a linear map $\phi: K\to V$, where $K$ is a subspace of $V$. The set of all partial endomorphisms of $V$ will be denoted by {\sf PEnd}$(V)$. Given   $\phi_1: K_1\to V$ and $\phi_2: K_2\to V$ in $\pen(V)$ and $\alpha\in \F$, we define the following operations:
\begin{align*}
   \phi_1+\phi_2 &: K_1\cap K_2\to V, \,x\mapsto \phi_1(x)+\phi_2(x); \\
   \phi_1\phi_2 &: \phi_2^{-1}(K_1)\to V, \, x\mapsto \phi_1(\phi_2(x));\\
   \alpha\phi_1 &: K_1\to V, \, x\mapsto \alpha\phi_1(x).
\end{align*}
Observe that the sum and the product of  $\phi_1$ and $\phi_2$ are defined on the largest possible domains  where the expressions $\phi_1(x)+\phi_2(x)$ and $\phi_1(\phi_2(x))$, respectively,  make sense. 
\begin{prop}\label{p20}
Let $V$ be a vector space. Then,  $\pen(V)$ with the above defined operations  is a right distributive associative inverse semialgebra.
\end{prop}
\begin{proof}
First, we show that $(\pen(V),+)$ is a commutative inverse monoid. Evidently, the sum is associative and commutative, and the  zero map $0: V\to V$ is the neutral element. Now, for $\phi:K\to V\in \pen(V)$  it is clear that its (additive) inverse is given by $-\phi:K\to V, \, x\mapsto -\phi(x)$. 

 For the rest of the proof let  $\phi_1: K_1\to V,$ $\phi_2: K_2\to V$ and $\phi_3: K_3 \to V$  be in  $\pen(V).$

  It is easy check that $\pen(V)$ is an inverse semivector space. Indeed,  to see (1) of Definition~\ref{def2}  observe that $\dom (\alpha  (\phi_1+\phi_2))=K_1\cap K_2=\dom (\alpha  \phi_1+ \alpha  \phi_2)$, and for $x\in K_1\cap K_2$ the equality  $(\alpha  (\phi_1+\phi_2))(x)=(\alpha  \phi_1)(x)+(\alpha \phi_2)(x)$ is immediate. Similarly, the equalities 
$$\dom ((\alpha +\beta)  \phi_1)=K_1=\dom (\alpha  \phi_1+ \beta  \phi_1) \;\;\; \text{and} \;\;\; 
\dom ((\alpha \beta)  \phi_1)=K_1=\dom (\alpha (  \beta  \phi_1))$$ readily give us (2) and (3) of Definition~\ref{def2}, respectively. Finally, it is obvious that $1_{\mathbb F}\phi_1=\phi_1,$ so that  (4) of Definition~\ref{def2} also holds.

		
Notice that $\E(\pen(V))=\{0_\phi\mid \phi\in \pen(V)\}$, and $\phi_1\preceq \phi_2\Longleftrightarrow \phi_1$ is a restriction of $\phi_2$. Since the composition of partial endomorphisms is associative, it remains to  verify that the inverse semivector space $\pen(V)$ satisfies Definition~\ref{def4} and that $\pen(V)$ is right distributive:
\begin{itemize}
\item[(1)] For  $ 0\neq \alpha \in \mathbb{F}$ we have $$\dom (\alpha (\phi_1 \phi_2))=\dom ((\alpha \phi_1)\phi_2)=\dom (\phi_1 (\alpha \phi_2)))=\phi_2^{-1}(K_1),$$
and clearly 
 $\alpha (\phi_1\phi_2)(x)=(\alpha \phi_1)(\phi_2 (x))=  \phi_1((\alpha \phi_2)(x))$ for all $x\in \phi_2^{-1}(K_1).$
\item[(2)] Note that $\dom(\phi_1 ( \phi_2+\phi_3))= (\phi_2+ \phi_3)^{-1} (K_1)$ and $ \dom
             (\phi_1\phi_2)\cap \dom ( \phi_1\phi_3)=
                          \phi_2^{-1}(K_1)\cap \phi_3^{-1}(K_1).$ Then $\dom (\phi_1\phi_2+ \phi_1  \phi_3 )\subseteq \dom\phi_1( \phi_2+\phi_3) $, and for all $x\in \phi_2^{-1}(K_1)\cap \phi_3^{-1}(K_1)$ we have 
  \begin{align*}
      ( \phi_1  (\phi_2+\phi_3))(x)=       
        (\phi_1 \phi_2)(x)+( \phi_1\phi_3) (x),     \end{align*}
yielding $\phi_1( \phi_2+\phi_3)\succeq \phi_1\phi_2+\phi_2  \phi_3.$
       \item[$(3)$] Observe that $$\dom (( \phi_1+\phi_2)\phi_3 )= \phi_3^{-1} (K_1\cap K_2)=\phi_3^{-1}(K_1)\cap\phi_3^{-1}(K_2)=\dom(\phi_1\phi_3+\phi_2\phi_3),$$ and for all $x\in\phi_3^{-1} (K_1\cap K_2)$ we have 
       \begin{align*}
      ((\phi_1+\phi_2)\phi_3)(x)=(\phi_1 \phi_3)(x)+( \phi_2\phi_3) (x).      
                \end{align*}
 Consequently, $ ( \phi_1+\phi_2)\phi_3= \phi_1\phi_3+\phi_2\phi_3,$ proving  \eqref{RightDistr}.  In particular, (3) of Definition~\ref{def4} holds. 


\item[$(4)$] 
 It is easily seen that, $$\dom (\phi_1(0_{\phi_3} +\phi_2))=(0_{\phi_3} +\phi_2)^{-1}(K_1)=\{x\in K_3\cap K_2\mid \phi_2(x)\in K_1\}=K_3\cap (\phi_2)^{-1}(K_1),$$
and $$\dom (\phi_10_{\phi_3} +\phi_1\phi_2)= 0_{\phi_3}^{-1}(K_1)\cap(\phi_2)^{-1}(K_1) =K_3\cap \phi_2^{-1}(K_1).$$ 
Thus $\dom (\phi_1(0_{\phi_3} +\phi_2))=\dom (\phi_10_{\phi_3} +\phi_1\phi_2)$. As 
        $$\phi_1 ((0_{\phi_3}+\phi_2)(x))= \phi_1  (\phi_2 (x)) =  \phi_1 (0_{\phi_3} (x))+ \phi_1(\phi_2(x)),$$ we conclude that $\phi_1 (0_{\phi_3}+\phi_2)=\phi_10_{\phi_3} +\phi_1\phi_2.$
        
\item[$(5)$] This item follows from the right distributivity property \eqref{RightDistr}.
\item[$(6)$] Clearly,   $\dom (0_{\phi_1}  0_{\phi _2})=  0_{\phi_2}^{-1}(K_1) = K_2$ and  
so that  $$  \dom (0_{\phi_1} +0_{\phi _2})  = K_1 \cap K_2 \subseteq  K_2= \dom ( 0_{\phi_1}  0_{\phi _2})=  \dom (0_{\phi _2}),$$  resulting in 
$$ 0_{\phi_1} +0_{\phi _2}  \preceq    0_{\phi_1}  0_{\phi _2}=  0_{\phi _2},  $$ 
in particular,  (6) of Definition~\ref{def4} holds.

\end{itemize}  
\end{proof}

 Observe that $\pen(V)$  satisfies the following specific property:
$$\phi_10_{\phi_2}=0_{\phi_2},$$ for all  $ \phi_1, \phi_2 \in \pen(V).$
 Indeed,  using the above notation,  
 $\dom (\phi_1 0_{\phi_2}) =0_{\phi_2}^{-1}(K_1)=K_2=\dom (0_{\phi_2})$, and 
$ \phi_1 0_{\phi_2}(x)= \phi_1 (0)= 0=0_{\phi_2}(x),  $ for all $x\in K_2,$ 
 as desired.

\subsection{Lie inverse semialgebras}\label{subsec:LieInv}
\begin{defn}\label{def6}
   A {\em Lie  inverse semialgebra} over the field $\mathbb{F}$ is an inverse semivector space  $S$ over the field $\mathbb{F}$ with a map  
	$[\cdot,\cdot]:S \times S \rightarrow S$ given by $(x,y)\mapsto [x, y]$  such that for every $x,y,z\in S$ and  $0\neq \alpha\in \F$
   \begin{enumerate}
     
     \item $ \alpha [x,y]= [\alpha x,y]= [x,\alpha y];$
     \item $[x,y+z]\succeq [x,y]+[x,z];$
     \item $[x,y]=-[y,x];$
     \item  $[x,e+z]= [x,e]+[x,z]$ for all $e\in \E(S);$
     \item $J(x,y,z):=[x,y,z]+[y,z,x]+[z,x,y]\preceq 0_{x+y+z};$
     \item $e+f=[e,f] $ for all $e,f\in \E(S).$
   \end{enumerate}
\end{defn}
 The operation $[\cdot,\cdot]$ in Definition~\ref{def6} will be called the Lie bracket.

 \begin{rem}  It might be interesting for the reader to compare our Definition~\ref{def6} with the notion of a Lie $\preceq$-semialgebra with a preorder $\preceq$ given  in~\cite[Definition 3.4]{Chapman}, where the authors deal with semialgebras equipped with a negation map.
\end{rem}

\begin{rem}\label{rem:LieInverseOk} Any Lie  inverse semialgebra over the field $\mathbb{F}$ is a non-associative inverse semialgebra over $\mathbb{F}.$ Indeed,    condition $(2)$ in  Definition \ref{def6} implies that $$-[x,y+z]\succeq -( [x,y]+[x,z])=- [x,y]-[x,z].$$ So, by $(3)$ of Definition \ref{def6}  we get $[y+z,x]\succeq [y,x]+[z,x],$ which gives (3) of Definition~\ref{def4}. Next,  combining~(3) and ~(4) of Definition~\ref{def6},  we see that $$[z+e,x]=-[x,z+e]=-[x,z]-[x,e]=[z,x]+[e,x],$$
which shows (5) of Definition~\ref{def4}. Finally, (6)
of Definition~\ref{def4} immediately follows from
(6) of Definition \ref{def6}. Thus, it is clear that if a 
 non-associative inverse semialgebra $S$, whose operation is denoted by the Lie bracket, satisfies  (3), (5) and (6)
 of Definition \ref{def6}, then $S$ is a Lie  inverse semialgebra. In this case we shall also say that   the non-associative inverse semialgebra $S$ is Lie.
\end{rem}

 A {\em Lie inverse subsemialgebra} of a Lie inverse semialgebra $S$ is an inverse semivector subspace of $S$ which is closed under the Lie bracket. A  {\em homomorphism} between two Lie inverse semialgebras $S$ and $T$ is an $\F$-linear map $\phi: S\to T$ such that $\phi([x,y])=[\phi(x),\phi(y)]$ for all $x,y\in S$. Evidently, the Lie inverse semialgebras with their homomorphisms form a category. 

The following lemma deals with some elementary properties of the Lie inverse semialgebras, which we will frequently use throughout the paper.

\begin{lem}\label{lemma-basicprop}
 Let $S$ be a Lie inverse semialgebra over the field $\mathbb{F}$. Then  for all $x,y,z\in S$ and $e \in \E(S)$, we have:
\begin{enumerate}
    \item $[x,-y]=- [x,y]=[-x,y];$
    \item $[x,e]\in \E(S);$ 
    \item $0_{[x,y]}\preceq [x,0_y]+[0_x,y]+0_{x+y};$
    \item If {\rm char} $\F \neq 2$, then $[x,x]\preceq 0_x$.
\end{enumerate}
\end{lem}
\begin{proof}

(1)  Using  $(1)$ of Definition \ref{def6}   and Lemma~\ref{p1f}, we obtain
  $$  [x,-y]= [x,(-1)y]=(-1)[x,y]=-[x,y].$$  
Similarly, $[-x,y]=-[x,y].$\\


\noindent (2) This directly follows from (2) of  Lemma~\ref{lemma_basicproperties}.\\


\noindent (3)  Item (3) of  Lemma~\ref{lemma_basicproperties} gives
$$0_{[x,y]} \preceq  [0_x,y]+[x,0_y]+  [0_x, 0_y] =  [0_x,y]+[x,0_y] + 0_{x+y},$$ since  $ [0_x, 0_y] =  0_x + 0_y =    0_{x+y}$ thanks to condition (6) of
 Definition~\ref{def6}.\\

\noindent (4) By item~(3) of Definition~\ref{def6}, we have $[x,x]+[x,x]=[x,x]-[x,x]=0_{[x,x]}\in \E(L)$. Thus $[x,x]=\frac{1}{2}0_{[x,x]}=0_{[x,x]}$, which,   by item~(3), implies  $[x,x]\preceq 0_{x+x} = 0_{2x} = 0_x.$
\end{proof}

\begin{rem} It  follows from (2) of Lemma~\ref{lemma-basicprop} and items (3) and (4) of  Definition~\ref{def6} that if $S$ is a  Lie inverse semialgebra over  $\mathbb{F}$ and $x,y \in S,$  then 
\begin{equation}\label{eq:ineq}
x \preceq  y \Rightarrow [x,z] \preceq [y,z]\;\;\text{and}\;\;  [z,x] \preceq [z,y]\;\;\; \forall z\in S.
\end{equation} Indeed, $x \preceq y$ means that $x= e+ y$ for some $e\in \E (S).$ Then $[x,z] = [e+y,z] = [e,z]+
[y,z] \preceq [y,z],$ and the second inequality is an
immediate  consequence of the first one.
\end{rem}

\begin{ex}
Every inverse semivector space $V$ has a trivial structure of a Lie inverse semialgebra, which is given by $[x,y]=0_{x+y}$ for any $x,y\in V$. 
\end{ex}

\begin{rem}\label{rem:InvAssocLieInequality} 
    If  $S$ is an  associative inverse semialgebra over the field
      $\mathbb{F},$ then
\begin{equation}\label{eq:InvAssocLieInequality}   
    0_{xy -yx} \preceq 0_x + 0_y,
\end{equation}  for all $x,y \in S.$ Indeed, using  (3) of  Lemma~\ref{lemma_basicproperties} and  (6) of Definition~\ref{def4}, we see that 
$$ 0_{xy - yx} =  0_{xy} + 0_{yx}\preceq 0_{x}0_{y} + 0_{y}0_{x} \preceq 0_y + 0_x,$$ as desired.
\end{rem}

\begin{prop}\label{p19}
Let $S$ be a right distributive associative inverse semialgebra over the field $\mathbb{F}.$ Then, the inverse semivector space $S$ endowed with the bracket $[x,y]=xy-y x$ is a Lie inverse semialgebra in which  the following identity holds
\begin{equation}\label{eq:Sminus}
0_{[x,y]}=[0_x,y]+[x,0_y],
\end{equation}  for all $x,y \in S.$
\end{prop}
\begin{proof}
We will verify conditions (1)-(6) of Definition~\ref{def6} for every $x,y,z\in S$, $e,f\in \E(S)$ and $0\neq \alpha\in \mathbb{F}$.

(1) The equalities  $ \alpha [x,y]= [\alpha x,y]= [x,\alpha y]$  follow immediately from condition~(1) of Definition~\ref{def4}.

(2) By~(2) and~(3) of Definition~\ref{def4} we have
\begin{multline*}
    [x,y]+[x,z]=xy-y x + xz-zx=(x y +  xz)-(yx +z x)\preceq  x(y+z)-(y+z)x  = [x,y+z].
\end{multline*}

(3) We see that $[x,y]=xy-y x =-(yx-xy )=-[y,x].$

(4) By condition~(4) and (5) of Definition~\ref{def4}  we obtain 
  $$ [x,e+z]=x(e+z)-(e+z) x=x e+ xz -e x-zx=[x,e]+[x,z].$$
     
(5) We must to show that $J(x,y,z)\preceq 0_x+0_y+0_z$. Since $S$ is right distributive we see that 
\begin{multline*}
    J(x,y,z)=(xy) z-(yx)z-z(x y-yx)+(yz) x-(z y)x-x(yz-zy)+(zx) y-(xz) y-y(z x-x z)=(*).
\end{multline*}
 Keeping in mind that $S$ is associative and applying   (2) of Definition~\ref{def4},  (3) of Lemma~\ref{lemma_basicproperties}, (6) of Definition~\ref{def4} and then \eqref{eq:InvAssocLieInequality} we obtain  
\begin{align*}
(*)&=  x(y z) -x(z y) -x(yz-zy) + y (z x) -y(xz) -y(z x-x z) +z(x y) -z (yx) -z(x y-yx)\\
&\preceq x[y,z] -x[y,z] + y [z, x]  -y[z, x] +z[x, y]  -z[x, y]  = 0_{x[y,z]}+0_{y[z,x]}+0_{z[x,y]}\\ & \preceq 0_{x} 0_{[y,z]}+0_{y}0_{[z,x]}+0_{z} 0_{[x,y]} \preceq  0_{[y,z]}+ 0_{[z,x]}+  0_{[x,y]} \preceq 0_x+0_y+0_z=0_{x+y+z}.
\end{align*}

(6)  By (2) and (1)  of Lemma~\ref{lemma_basicproperties}, we obtain
$$[e,f] = ef-fe = ef+fe = e+f, $$
for all $e,f \in \E(S).$

For the last assertion, note that    (3) of  Lemma~\ref{lemma_basicproperties} and (1) of  Lemma~\ref{lemma_distributive}   imply 
$$0_{xy} = 0_{xy} + x 0_y =  0_{x}y + x 0_y $$ and similarly,  $0_{yx} = 0_y x + y 0_x.$ Consequently, using  (2) of  Lemma~\ref{lemma_basicproperties}, we conclude that
\begin{align*} 
0_{[x,y]}&=0_{xy - yx}=   0_{xy}+0_{yx}=  0_{x}y + x 0_y +  0_y x + y 0_x\\
& = 0_{x}y - y 0_x + x 0_y -  0_y x  =
[0_x,y]+[x,0_y],
\end{align*} 
as required.
\end{proof}

The Lie inverse semialgebra of the Proposition~\ref{p19} will be denoted by $S^-$. As an immediate consequence of Propositions~\ref{p19} and~\ref{p20}, we have:

\begin{cor}
Let $V$ be a vector space. Then the set $\pen(V)^-$ of all partial endomorphism of $V$ is a Lie inverse semialgebra.
\end{cor}

\begin{defn}
Let $A$ be a non-associative algebra. A {\em partial derivation } of $A$ is a linear map $\phi : K \rightarrow A$, where $K$ is a subalgebra of $A$ and 
$\phi (x y)=\phi(x)y+x\phi(y)$
 for every $x, y \in K$. The set of all partial derivations of $A$ will be denoted by $\pde(A)$.
\end{defn}
\begin{prop}
Let $A$ be a non-associative algebra. Then, $\pde(A)$ is a  Lie inverse subsemialgebra of $\pen(A)^-$.
\end{prop}
\begin{proof}
Given $\alpha \in \F$ and $\phi_i: K_i\to A\in \pde(A)$, $i=1,2$, we observe that
$\dom(\alpha \phi_1) = K_1,$  $\dom(\phi_1+\phi_2)=K_1\cap K_2$ and $\dom([\phi_1,\phi_2])=\phi_1^{-1}(K_2)\cap \phi_2^{-1}(K_1) \subseteq K_1 \cap K_2$ are subalgebras of $A$.
It is immediate to check that  $\alpha \phi_1 \in \pde(A)$ and $ \phi_1+\phi_2\in \pde(A),$ so that 
$\pde(A)$ is an inverse semivector subspace of $\pen(A).$ It remains  to show that $\pde(A)$ is closed under  the Lie bracket of $\pen(A)^-$, i.e. 
 $[\phi_1,\phi_2]=\phi_1\phi_2-\phi_2\phi_1\in \pde(A).$  For note that for all $x,y\in \phi_1^{-1}(K_2)\cap \phi_2^{-1}(K_1)$ we have 
$$\phi_1\phi_2(xy)=\phi_1 \big{(} \phi_2(x) y+ x\phi_2(y)\big{)}= (\phi_1\phi_2(x))y+\phi_2(x)\phi_1(y)+\phi_1(x)\phi_2(y)+x(\phi_1\phi_2(y)),$$
and interchanging $\phi_1$ and  $\phi_2$
$$\phi_2\phi_1(xy)=\phi_2 \big{(} \phi_1(x) y+ x\phi_1(y)\big{)}= (\phi_2\phi_1(x))y+\phi_1(x)\phi_2(y)+\phi_2(x)\phi_1(y)+x(\phi_2\phi_1(y)).$$
Thus $[\phi_1,\phi_2](xy)=[\phi_1,\phi_2](x)y+x[\phi_1\phi_2](y)$, as required.
\end{proof}

We present another pivotal example of a Lie inverse semialgebra, which will be frequently used in the subsequent sections.  It  is inspired by the inverse semigroup $S(G)$  introduced by R. Exel  in  \cite{Exel1998} to govern partial  representations of a group $G$ and partial actions of $G$ on sets.  It is known that $S(G)$ is isomorphic to the Szendrei expansion of $G,$ which is isomorphic to  the Birget–Rhodes prefix expansion of $G$ (see \cite{KL}, \cite{Szendrei}).

Let $L$ be a Lie algebra.  Let us consider the set 
    $$E(L):=\{(A,a) \mid A ~~\text{is a finite-dimensional subspace of }~~L,~~ a\in A\}.$$
In $E(L)$, we define the sum and the Lie bracket of two elements $(A,a)$ and $(B,b)$ as follows: 
$$(A,a)+(B,b)=(A+B,a+b) \quad \mbox{ and } \quad [(A,a), (B,b)]=(A+B+\mathbb{F}[a,b],[a,b]).$$
Additionally, for $\alpha\in \F$ we define:
$$\alpha (A,a)=(A,\alpha a).$$ 
It is straightforward to verify that $(E(L),+)$ is an inverse semivector  space with zero element  $(0,0)$, in which the (additive) inverse of $(A,a)$
is  $(A,-a).$  Moreover, for any $(A,a)$ we have that $0_{(A,a)}=(A,0)$, and so $$\E(E(L))=\{(A,0)\mid A \mbox{ is a finite-dimensional subspace of } L \}.$$ Consequently, $ (A,a)\preceq (B,b)$ holds if and only if 
 \begin{align*}
 (A,a)+ (A,-a)+(B,b)=(A,a)  \Longleftrightarrow (A+ B,b)=(A,a) \Longleftrightarrow a=b ~~\text{and}~~ B\subseteq A.
 \end{align*}

  \begin{prop}\label{p: E(L)}
     $E(L)$ endowed with the above operations is a Lie inverse semialgebra,  such that
     \begin{equation}\label{eq:E(L)}
[0_x, y]=[x,0_y]=0_{x+y},  
\end{equation}
for all $x,y \in E(L).$
  \end{prop}
  \begin{proof}
 We verify conditions~(1)-(6) of Definition~\ref{def6} for all $(A,a), (B,b), (C,c)\in E(L)$ and $0\neq\alpha\in \F $.
 
      (1) We have
$$ \alpha [(A,a), (B,b)]=\alpha  (A+B+\mathbb{F}[a,b],[a,b])= (A+B+\mathbb{F}[a,b],\alpha [a,b]).$$
On the other hand  $$[\alpha (A,a), (B,b)]=[ (A,\alpha a), (B,b)]=  (A+B+\mathbb{F}[\alpha a,b],[\alpha a,b])= (A+B+\mathbb{F}[a,b],\alpha [a,b]).$$
Thus, we have shown that $\alpha  [ (A,a), (B,b)]=[\alpha (A,a), (B,b)].$ Similarly, we can check the equality $\alpha  [ (A,a), (B,b)]=[ (A,a), \alpha (B,b)]$.

(2) We see  that 
$$[(A,a), (B,b)+(C,c)]=  [(A,a), (B+C,b+c)]= (A+B+C+\mathbb{F}[a,b+c],[a,b+c]).$$
On the other hand
\begin{align*}
[(A,a), (B,b)]+ [(A,a),(C,c)] &=  (A+B+\mathbb{F}[a,b],[a,b])+(A+C+\mathbb{F}[a,c],[a,c])\\
&=(A+B+C+\mathbb{F}[a,b]+\mathbb{F}[a,c],[a,b+c]).
\end{align*}
As $ \mathbb{F}[ a,b+c] \subseteq \mathbb{F}[a,b]+\mathbb{F}[a,c]$, we obtain $ [(A,a), (B,b)+(C,c)] \succeq [(A,a), (B,b)]+ [(A,a),(C,c)].$

(3)  This is straightforward from the definition of the Lie bracket in $E(L).$

(4) We compute  that 
\begin{align*}
       [(A,a),(B,0)+(C,c)]&=[(A,a),(B+C,c)]= (A+B+C+\mathbb{F}[a,c],[a,c])\\
       &=  (A+B,0)+ (A+C+\mathbb{F}[a,c],[a,c])\\
       &= [(A,a),(B,0)]+[(A,a),(C,c)],
\end{align*} as desired.

(5) Since  $$[(A,a), (B,b),(C,c)]=[(A+B+\mathbb{F}[a,b],[a,b]),(C,c)]=(A+B+C+\mathbb{F}[a,b]+\mathbb{F}[a,b,c],[a,b,c]),$$ we obtain 
\begin{align*}
    J((A,a), (B,b),(C,c))&=(A+B+C+\mathbb{F}[a,b]+\mathbb{F}[b,c]+\mathbb{F}[c,a]+\mathbb{F}[a,b,c]+\mathbb{F}[b,c,a]+\mathbb{F}[c,a,b],0)\\
    & \preceq (A+B+C,0)=0_{(A,a)+(B,b)+(C,c)}.
\end{align*}

(6) This immediately follows  from the definition of the Lie bracket in $E(L)$.

 Finally,  for any $(A,a), (B,b)\in E(L)$ we have that
$$[0_{(A,a)}, (B,b) ] = [(A,0), (B,b) ]=
(A+B,0) =  [(A,a), 0_{(B,b)} ],$$ showing \eqref{eq:E(L)}.  
 \end{proof}

\begin{ex}\label{exHeisenberg}
Let $L$ be the Heisenberg 3-dimensional Lie algebra with basis  $\{a,b,c\}$,  where $[a,b]=c$ and $[a,c]=[b,c]=0$. Let $A= \F a$ and $B= \F b$. Then, 
$$0_{[(A,a),(B,b)]}=(L,0)\neq (A+B,0) =  [0_{(A,a)},(B,b)]+[(A,a),0_{(B,b)}].$$ 
Thus, $x=(A,a), y= (B,b), z= (C,c)$  do not satisfy \eqref{eq:Sminus} and by Proposition~\ref{p19}, $E(L)$ cannot be embedded in $S^-$ for a right distributive associative inverse semialgebra $S$.  Note also that 
$$0_{[x,y]} = (L,0)\neq (A+B,0) =  0_{x+y}. $$ 
\end{ex}

In the next section, we will see that certain types of Lie inverse semialgebras can be embedded into some $S^-$
for an appropriate distributive associative inverse semialgebra.

\section{Semilattices of algebras}\label{s: semillatices}
Let $\Lambda$  be a poset. If $\lambda$ and $\mu$ are elements of $\Lambda$, then their  {\em meet} is an element $\lambda\wedge \mu$ such that
\begin{enumerate}
    \item $\lambda\leq \lambda\wedge\mu$ and $\mu\leq \lambda\wedge\mu$, 
    \item if $\nu\leq \lambda$ and $\nu\leq \mu$, then $\nu\leq \lambda\wedge\mu.$
\end{enumerate}
A poset $\Lambda$ is called a {\em meet semilattice} if every finite subset of $\Lambda$ has a meet. Equivalently, a meet semilattice is a commutative semigroup $(\Lambda,\wedge)$ in which every element is idempotent.  For instance, the set of idempotents of an inverse semigroup  is a meet semilattice.  In particular, $\E (V)$ is a meet semillatice, where $V$ is an inverse semivector spaces.

A meet semilattice $\Lambda$ can be also viewed as a category, where the objects are the elements of $\Lambda$ and the arrows are given by the ``less than or equal to'' relations.  A presheaf on $\Lambda$ with values in the category $\mathcal{C}$ is a contravariant functor $F: \Lambda\to \mathcal{C}$.
The main objective of this section is to show that the presheaf on a meet semilattice, with values in the category of associative (or Lie) algebras, corresponds to a certain subclass of associative (or Lie) inverse semialgebras.

\begin{defn}\label{def:SemilatNonAssocAlg} A {\em semilattice of non-associative algebras} over  the field $\F$ is a  non-associative inverse semialgebra $S$ over $\F$ such that 
\begin{equation}\label{eq:semilatticeAlg}
    0_{xy}=0_{x+y}
\end{equation} for any  $x,y \in S.$ If, in addition, $S$ is  associative (or Lie) then  $S$ will be called a {\em semilattice of associative (or Lie) algebras}.
\end{defn}

Condition \eqref{eq:semilatticeAlg} implies distributivity:

\begin{lem}\label{lemma1_SemilatticeAlg}
If $S$ is a semilattice of non-associative algebras over
$\F$ then $S$ is distributive. 
 \end{lem}
 \begin{proof}
  Let $x,y,z\in S.$ Since $xy+xz\preceq x(y+z)$,  we have, using \eqref{eq:semilatticeAlg}, that $$xy+xz= 0_{xy}+0_{xz}+x(y+z)=0_{x+y}+ 0_{x+z}+x(y+z)= 0_{x(y+z)}+ x(y+z)=x(y+z).$$ 
Analogously we check that $(y+z)x=yx+zx.$
\end{proof}

The behavior of the (additive) idempotents also becomes better. Indeed,  Lemma~\ref{lemma1_SemilatticeAlg}, Lemma~\ref{lemma_distributive} and \eqref{eq:semilatticeAlg} directly imply:

\begin{lem}\label{lemma2_SemilatticeAlg} 
Let $S$ be a semilattice of non-associative algebras over
$\F .$ Then 
$$ x 0_y = 0_x y = 0_x 0_y = 0_{xy} = 0_x + 0_y$$  
for all $ x,y \in S.$ In particular, 
$$ef=e+f = fe $$ 
for all $\, e,f \in \E (S).$
     \end{lem}

To be aligned with our notation  used in the case of Lie inverse semialgebras, we use brackets for the product in semilattices of Lie algebras.  In this case we have the next:
\begin{lem}\label{lemma_SemilatticeLieAlg} Let $S$ be a semilattice of Lie algebras over  $\F .$ Then for all $x, y , z\in S$
    $$  J(x,y,z) = 0_{x+y+z}. $$
    Moreover,  if {\rm char} $ \F \neq 2$, then 
    $$[x,x] = 0_x .$$
\end{lem}
\begin{proof} In the Lie case, \eqref{eq:semilatticeAlg} takes the form $$0_{[x,y]} = 0_{x + y},$$ which readily implies that 
    $$0_{[x,y,z]} =  0_{x} +0_{[y,z]} =0_{x} +0_y + 0_z $$ for all $x,y, z\in S.$ Then, since 
    $J(x,y,z)\preceq 0_{x+y+z},$ we obtain
$$J(x,y,z) = J(x,y,z) + 0_{x+y+z}=
0_{[x,y,z]} + 0_{[y,z,x]}+0_{[z,x,y]} + 0_{x+y+z}
=0_{x} +0_y + 0_z  + 0_{x+y+z} = 0_{x+y+z},$$
proving   $J(x,y,z) = 0_{x+y+z}.$ 

Suppose that  {\rm char} $ \F \neq 2 .$ Then by  (4) of Lemma~\ref{lemma-basicprop}, $[x,x] \preceq 0_x $ and, consequently,  
$$[x,x] = 0_{[x,x]} + 0_x  = 0_x,$$ as desired.
\end{proof}

\begin{rem}\label{rem:AssocToLieSamillatAlg} Let  $S$ be a semilattice of associative algebras. Then 
 \eqref{eq:semilatticeAlg} readilty implies that  
 $0_{[x,y]} = 0_{x + y},$ so by 
 Lemma~\ref{lemma1_SemilatticeAlg} and  
 Proposition~\ref{p19}, we obtain that  $S^-$ is a semilattice of Lie algebras. 
\end{rem}

Evidently, the collection of semilattices of associative (or Lie) algebras constitutes a subcategory within the category of associative (or Lie) inverse semialgebras. 


\begin{ex}\label{example:Sheaves}
 Let $\mathcal F$ be a presheaf of associative $\mathbb{F}$-algebras over a topological space $(X, \tau).$  Given $U,V \in \tau$, with $V \subseteq U$ denote by  
${\textrm res}^U _{V}: \mathcal F (U) \to \mathcal F  (V)$ the restriction homomorphism. Take the  disjoint union
$$S_{\mathcal F} =    \dot{\bigcup}  _{U \in \tau} \mathcal F (U) $$ and define  the  following operations on $S_{\mathcal F}:$ for $U,V \in \tau ,$  $x\in \mathcal F (U), y \in \mathcal F (V)$ and  $\alpha \in \mathbb{F} $ put
$$\alpha \cdot x :=  
 \alpha  x \in {\mathcal F}(U),$$
 $$x+y :=  {\textrm res}^U _{U\cap V}(x) + 
 {\textrm res}^V _{U\cap V} (y) \in {\mathcal F}(U\cap V),$$
 $$x \cdot y :=  {\textrm res}^U _{U\cap V}(x) \cdot 
 {\textrm res}^V _{U\cap V} (y)\in {\mathcal F}(U\cap V).$$ Then it is easy to see that  $S_{\mathcal F} $  is a semilattice of associative $\mathbb{F}$-algebras with the above defined operations.
    The following are  particular cases of this construction: 
\begin{enumerate}
\item $S_{\mathcal F}$ is the  set  of all sections of the structural sheaf ${\mathcal F}$ of a scheme $X$ over  $\mathbb{F}$ .\\

\item  $S_{\mathcal F}$ is the  set  of all regular functions $U \to \mathbb{F},$ where $U$ runs over all open subsets of an algebraic variety $X$ over   $\mathbb{F};$ ${\mathcal F}$ is the sheaf of regular functions whose restriction maps are the usual restrictions of functions. \\

\item 
$S_{\mathcal F}= {\mathcal F}_{par}(X,\R )$ is the set 
 of all continuous functions $f:U\to \R,$ where $U$ runs over all open subsets of a topological space $X.$ Such function form a sheaf on $X$ whose restriction maps are the usual restrictions of functions. 
 Replacing $\R$ by $\C$ we obtain a similar semilattice ${\mathcal F}_{par}(X,\C )$ of associative algebras over $\C .$ In particular, if $X$ is discrete, then ${\mathcal F}_{par}(X,\R )$ (or  ${\mathcal F}_{par}(X,\C ))$ consists of all partial functions (functions defined on subsets of $X$) with values in $\R$ (or $\C$). \\
 

\item Let $X$ be a smooth manifold. Then the set of all smooth functions $U \to \R,$ where $U$ runs over the open subsets of $X$ (considered as embedded open submanifolds) form  a semilattice of associative algebras over $\R$ with the operations defined using the corresponding sheaf of functions.\\

\item Let $X$ be a complex manifold. Then the set of all holomorphic  functions $U \to \C,$ where $U$ runs over the open subsets of $X$ form  a semilattice of associative algebras over $\C$ with the operations defined using the corresponding sheaf of functions.\\

\item Let $X$ be a topological space. For any open subset $U$ of $X$ let ${\mathcal F} (U) $ be the $C^*$-algebra   $C_b(U )$ of the   bounded continuous functions $U\to \C$ and ${\mathcal F} $ be the corresponding sheaf with the usual restriction of functions. Then $S_{\mathcal F}$ is  a semilattice of associative algebras over $\C$  (this is an example of what could be called a semilattice of $C^*$-algebras).\\

\end{enumerate}

\end{ex}

\begin{ex}\label{example:SheavesLie}
 Let $\mathcal F$ be a presheaf of associative $\mathbb{F}$-algebras over a topological space $(X, \tau)$ and
 $S_{\mathcal F}$ be the semilattice of associative algebras from  Example~\ref{example:Sheaves}.  Then  by Remark~\ref{rem:AssocToLieSamillatAlg}, $S_{\mathcal F}^{-}$ is  a semilattice of Lie algebras. Clearly,
$$S_{\mathcal F}^- =    \dot{\bigcup}  _{U \in \tau} \mathcal F (U)^-, $$ where each ${\mathcal F} (U)^-$ is the Lie algebra obtained by endowing ${\mathcal F} (U)$ with the Lie bracket
$[x,y]= xy - yx.$ In particular, items (1)-(7) from Example~\ref{example:Sheaves} produce semilattices of Lie algebras.
\end{ex}

\begin{ex}
If $L$ is an abelian Lie algebra, then it is clear that $E(L)$ is a semilattice of Lie algebras. Similarly, if $L$ is a 2-dimensional solvable Lie algebra, $E(L)$ is also a semilattice of Lie algebras.  Indeed,  for all $(A,a)$ and $(B,b)$ in $E(L)$ we need to check that
$$A+B+\mathbb{F}[a,b]=A+B.$$
If either $A$ or $B$ is a trivial subspace of $L$, the equality is immediate. Now, assume that $A$ and $B$ are non-trivial subspaces of $L$. If $A\neq B$, we have $A+B=L$, and the equality still holds. Finally, if $A=B$, then $[a,b]=0$ for any $a,b\in A$, as $A$ is 1-dimensional. Thus, the equality holds for every $(A,a),(B,b)\in E(L)$.

On the other hand, if $L$ is the 3-dimensional Heisenberg algebra, then it follows from Example~\ref{exHeisenberg} that $E(L)$ is not a semilattice of Lie algebras. 
\end{ex}

To provide more examples of semilattices of Lie algebras, we introduce the following notation. Let $\mathcal{A}$ be a class of non-associative algebras, and let $A\in \mathcal A$. Let us consider the set  
$$\pde_{\mathcal A}(A)=\{\phi: I\to A\in \pde(A)\mid I\unlhd A \mbox{ and } I\in\mathcal{A}\},$$  where
$I\unlhd A$ denotes the fact that $I$ is a two-sided ideal of $A.$
Before stating the following proposition, which offers examples of semilattices of Lie algebras, we first recall that a Lie algebra $L$ is called {\em sympathetic} if it is perfect, centerless, and has no outer derivations, i.e. $[L,L]=L, Z(L)=0$ and $\der(L)=\ad (L)$. Every semisimple Lie algebra is sympathetic, but there are sympathetic Lie algebras that are not semisimple. 

\begin{prop}\label{u1}
Suppose that $\mathcal A$ is any of the following classes: unital associative algebras, unital Jordan algebras,  sympathetic Lie algebras or semisimple Lie algebras.  Let $A\in \mathcal A$. Then, for  $\phi_i:K_i\rightarrow A$ in $ \pde_{\mathcal{A}}(A), i=1,2,$ we have 
$$\dom(\phi_1+\phi_2)=K_1\cap K_2=\phi_1^{-1}(K_2)\cap \phi_2^{-1}(K_1)=\dom ([\phi_1,\phi_2]).$$
Consequently, $\pde_{\mathcal{A}}(A)$ is a  Lie inverse subsemialgebra of $\pde(A)$ which is a semilattice of Lie algebras. 
\end{prop}
\begin{proof}
First observe that the intersection of two unital ideals of an associative algebra is unital and that every ideal of a semisimple Lie algebra is itself semisimple. Furthermore, if $\mathcal{A}$ is the class of sympathetic Lie algebras or unital Jordan algebras, then  \cite[Corollary 2.11]{Cortes2023} and~\cite[Proposition 2.14(1)]{Cortes2023} imply that $K_1\cap K_2$ belongs to $\mathcal{A}$. Therefore, in all classes $K_1\cap K_2$ is idempotent. Let $x\in K_1\cap K_2.$ As $x=yz$ for some $y,z\in K_1\cap K_2$, we have
$$\phi_i(x)=\phi_i(yz)=\phi_i(y)z+y\phi_i(z)\in K_1\cap K_2\subseteq K_i,$$ for $i=1,2.$  Consequently, $\phi_1^{-1}(K_2)\cap \phi_2^{-1}(K_1)\supseteq K_1\cap K_2$, and since the opposite inclusion is obvious, we obtain the equality $K_1\cap K_2=\phi_1^{-1}(K_2)\cap \phi_2^{-1}(K_1)$. Moreover, as  $\dom(\phi_1+\phi_2)=\dom([\phi_1,\phi_2])\in \mathcal{A}$, it follows that $\pde_{\mathcal{A}}(A)$ is a  Lie inverse subsemialgebra of $\pde(A)$. Additionally,  $0_{\phi_1+\phi_2}=0_{[\phi_1,\phi_2]}$. So, by Definition~\ref{def:SemilatNonAssocAlg},   $\pde_{\mathcal{A}}(A)$ is a semilattice of Lie algebras.
\end{proof}

Let $\Lambda$ be a meet semilattice,  $\mathcal{C}$ be the  category of associative (or Lie)  algebras 
and let $F: \Lambda\to \mathcal{C}$ be a presheaf. For each object $\lambda$ in $\Lambda$ let $F(\lambda)=S_\lambda$, and for any morphism  $f:\mu\to \lambda$ in $\Lambda$ let $F(f)=\eta_{\lambda,\mu}: S_\lambda\to S_\mu$. The collection $\{\eta_{\lambda,\mu}\mid \lambda\geq \mu\}$ satisfy 
\begin{enumerate}
    \item $\eta_{\lambda,\lambda}=id_{V_{\lambda}}$, and
    \item $ \eta_{\mu,\nu} \eta_{\lambda,\mu}= \eta_{\lambda,\nu}$ for $\lambda \geq\mu \geq \nu$.
\end{enumerate}
It is straightforward to verify that 
$S_F=\dot{\bigcup}_{\lambda\in \Lambda} S_{\lambda}$ endowed with the operations
    \begin{align*}
      x_{\nu}+x_{\mu}&=\eta_{\nu,\lambda}(x_{\nu})+ \eta_{\mu,\lambda}(x_{\mu})\\
      x_{\nu}x_{\mu}&=\eta_{\nu,\lambda}(x_{\nu})\eta_{\mu,\lambda}(x_{\mu}),
         \end{align*}
where $x_\nu\in S_\nu, x_\mu\in S_\mu$, and $\lambda=\nu\wedge\mu$, is a semilattice of 
  associative (or Lie) algebras.  Thus, for a given presheaf $F:\Lambda\to \mathcal{C}$, we obtain a semilattice of associative (or Lie) algebras $S_F.$  This establishes a correspondence $F\to S_F$ from presheaves to  semilattices of associative (or Lie) algebras. For a fixed $\Lambda$ a straightforward verification shows that this correspondence gives a functor from the category of presheaves of associative (or Lie) algebras on  $\Lambda $  to the category of semilattices of associative (or Lie) algebras.

\begin{prop}\label{prop-semila}
 Let $S$ be a semilattice of  associative (or Lie) algebras. Then there exists a presheaf $F(S)$ such that $S$ is isomorphic to $S_{F(S)}=\dot{\bigcup}_{\lambda\in \Lambda}  S_{\lambda}$. 
\end{prop}

\begin{proof}
We will prove the proposition in the associative case, the Lie case being similar. 
As $(S,+)$ is a commutative inverse semigroup, the set $\Lambda=\E(S)$ is a meet semilattice.  For each $\lambda\in \Lambda$, let us consider $S_{\lambda}=\{ x\in S\mid 0_x=\lambda\}$. Since for any $x,y\in S$ and $\alpha\in \F$, we have $0_{xy}=0_{x+y}=0_x+0_y$ and $0_{\alpha x}=0_x$, it follows that each $S_\lambda$ is an associative algebra, whose zero element is $\lambda$. For $\lambda\geq \mu$, the translation map $\eta_{\lambda,\mu}:S_{\lambda} \rightarrow S_{\mu}$, defined by $x\mapsto \mu+x,$ is  linear. We verify that $\eta_{\lambda,\mu}$ is also an algebra homomorphism. For let $x,y\in S_\lambda$, then
$$(\mu+x)(\mu+y)=\mu\mu+x\mu+\mu y+xy=
 \mu+\lambda+\mu+\lambda+\mu+xy=\mu+xy,$$  showing that $\eta_{\lambda,\mu}$ preserves the product. Moreover, if  $\lambda\geq \mu \geq \nu$,  then 
$\nu=\mu+\nu$ implies $ \eta_{\mu,\nu} \eta_{\lambda,\mu}= \eta_{\lambda,\nu}$.   This way we obtain  the contravariant functor $F:\Lambda\to \mathcal{C}$ defined by $F(\lambda)=S_\lambda$ and $F(f)=\eta_{\lambda,\mu}$, for a morphism $f:\mu \to \lambda$ in $\Lambda$, which is the claimed  presheaf $F(S).$ It is easy to see that $S$ is isomorphic to   $S_{F(S)}=\dot{\bigcup}_{\lambda\in \Lambda} S_{\lambda}.$ 
\end{proof}

Given  a semilattice of  associative (or Lie) algebras $S,$  let $S_{\lambda}$ be as in the proof of Proposition~\ref{prop-semila}, i.e.   $S_{\lambda}=\{ x\in S\mid 0_x=\lambda\}$, $(\lambda \in \Lambda).$ Then obviously $S$ is a disjoint union of the $S_{\lambda}$'s and identifying it with the external disjoint union of the $S_{\lambda}$'s, we may assume that  $S =S_{F(S)}.$ Then the correspondence $S\to F(S)$ is the inverse of the correspondence $F\to S_F$ defined above, so that there is a one-to-one correspondence between 
presheaves on meet semilattices of associative (or Lie) algebras and   semilattices of associative (or Lie) algebras.

 It is worth noting  that there are homomorphisms 
 of semilattices of associative (or Lie) algebras which do not  correspond to morphisms of presheaves.


\begin{ex} Let $\varphi : A \to B$ be a homomorphism of associative (or Lie) algebras and let $\Lambda $ be the semilattice 
$\{ \lambda, \mu \} $ in which  $ \lambda \geq \mu  $. Write $S_{\lambda} = A,$  $S_{\mu} = B$ and
$\nu _{\lambda, \mu} = \varphi ,  \nu _{\lambda, \lambda } = id _A, \nu _{\mu, \mu} = id _B.$ Then
$S= S_{\lambda} \dot{\cup} S_{\mu}$  is a semilattice of  associative (or Lie) algebras. Consider the map 
$\phi : S \to S$ defined by $\phi (a) = \varphi (a)$ and
$\phi (b) =b$  for any $a\in A, b\in B$. Then evidently
$\phi$ is a homomorphism of inverse semialgebras, which do not come from a morphism of presheaves.
 \end{ex}
 
 Let $\Lambda $ be a meet semilattice and consider the  category   $\mathcal D _{\Lambda}$  whose objects are the semilattices of non-associative algebras $S$ such that $\E (S) \cong \Lambda $. Identify  $\E (S)$ and $\Lambda $. Let the morphisms in $\mathcal D _{\Lambda}$  be those  homomorphisms of  semilattices of non-associative  algebras which restrict to the identity map on $\Lambda $.

\begin{prop}\label{prop:CategoryEquiv} Let $\Lambda $ be a meet semilattice.
    Then the category of sheaves of non-asociative algebras on $\Lambda $  is equivalent to the category
    $\mathcal D _{\Lambda}.$
\end{prop}
\begin{proof} The above constructed correspondence $S\to F(S),$ restricted to the
semilattices of non-associative algebras $S$ with $\E (S) \cong \Lambda ,$ determines a contravariant functor, which  together with the functor  $F\to S_F$  gives a category equivalence.
\end{proof}

\begin{rem} Evidently, the category equivalence in Proposition~\ref{prop:CategoryEquiv} restricts to an equivalence between the category of sheaves of asociative
    (or Lie)  algebras on $\Lambda $  and the category of semilattices of associative (or Lie) algebras $S$ such that $\E (S) \cong \Lambda $ with morphisms as in $\mathcal D _{\Lambda}.$ 
\end{rem}

\begin{cor}\label{c: semillatice}
Every semilattice $S$ of Lie algebras can be embedded into $T^-$ for some semilattice of associative algebras $T$. 
\end{cor}
\begin{proof}
Let $U$ be the covariant functor that sends a Lie algebra to its universal enveloping algebra. As, by Proposition~\ref{prop-semila}, 
 $S\cong S_{F(S)},$ 
the composition $UF$ is a contravariant functor
from the category of semilattices of Lie algebras to that of presheaves  of associative algebras on $\Lambda .$
For each $\lambda\in \Lambda$, let $\iota_\lambda:F(\lambda)\to U(F(\lambda))^-$ be the canonical injection. Then, setting
 $T=
\dot{\bigcup}_{\lambda\in \Lambda} U(F(\lambda))$, it is readily seen   that  the map $\iota:L\to T^-$, defined by $\iota|_{F(\lambda)}=\iota_\lambda$, is an injective homomorphism of semilattices of Lie algebras. 
\end{proof}

\begin{rem}Considering an inverse semivector space $V$ as a semilattice of associative (or Lie) algebras, where the product for all $x$ and $y$ is defined by $xy=0_{x+y}$, we obtain by Proposition~\ref{prop-semila} a presheaf $F(V)$ of vector spaces on $\Lambda .$  Then 
 $V=
\dot{\bigcup}_{\lambda\in \Lambda} F(\lambda)$, with
$F(\lambda )= V_{\lambda}=\{ x\in V\mid 0_x=\lambda\},$ which we call a {\it semilattice of vector spaces}. Thus,  any inverse vector  space $V$ is a semilattice of vector spaces. Obviously,  $(V,+)$  is a semilattice 
of abelian groups, in the sense which is well-established in the theory of semigroups. Clearly, any
non-associative  inverse semialgebra, with respect to the addition and multiplication by the scalars from the field $\F ,$ is a  semilattice of vector spaces.
\end{rem}


\section{Partial actions and premorphisms}\label{s: partial actions}
The main goal of this section is to introduce the notion   of a premorphism between Lie inverse  semialgebras,  that of a partial action of a Lie algebra $L$ on a non-associative algebra and to explore their relationship with the Lie inverse semialgebra \( E(L) \) introduced in  Section~\ref{subsec:LieInv}. 


 Partial group actions on sets can be defined using the concept of a premorphism of inverse semigroups (see \cite{KL}). For the Lie case we give the following:
\begin{defn}\label{def:premorphism}
Let  $S,T$ be two Lie inverse semialgebras. A map $\rho:S\rightarrow T$ is called a {\em premorphism} if for all $x,y\in S$ and $0\neq\alpha\in \F$ we have
\begin{enumerate}
\item $\rho(x+y)\succeq \rho(x)+\rho(y),$
         \item $\rho([x,y])\succeq [\rho(x),\rho(y)],$
         \item $\rho(\alpha x)=\alpha \rho(x)$.  
\end{enumerate}
If  $(S,+)$ and $(T,+)$ are monoids we also require that $\rho(0)=0.$  
\end{defn}

 Notice that (1) and (3) of Definition~\ref{def:premorphism} imply that 
$$\rho (0_x) \succeq 0_{\rho(x)},$$  for any $x\in S.$ Indeed, $ \rho (0_x) = \rho (x + (-1)x) \succeq 
 \rho (x)  + \rho((-1)x)=  \rho (x)  - \rho(x)=
 0_{\rho(x)}.$ Moreover,
if $\rho: S\to T$ is a premorphism of Lie inverse semialgebras with zero element, then, in particular, $\rho$ is a premorphism of the additive semigroups $(S,+)$ and $(T,+)$. Hence, as $(S,+)$ and $(T,+)$ are monoids, then, by~\cite[Proposition 2.1]{KL}, $\rho$ satisfies
\begin{equation}\label{eq-premorphism}
    \rho(x)+\rho(y)=\rho(x+y)+0_{\rho(y)},
\end{equation}
for all $x,y\in S$.

\begin{defn}\label{strong}
 Let $S$ and $T$ be Lie inverse semialgebras.  
A premorphism $\rho:S \rightarrow T$ is called {\em strong} if
$$ [\rho(x),\rho(y)]=\rho([x,y])+0_{\rho(x)+\rho(y)},
$$ for all $x,y \in L.$  
\end{defn}

One can easily see that  a premorphism  $\rho:S \rightarrow T$ is strong if and only if
\begin{equation}
0_{\rho([x,y])}+0_{\rho(x)+\rho(y)}=0_{[\rho(x),\rho(y)]},
\end{equation}
for all $x,y\in S$. For instance, if $T$ is a semilattice of Lie algebras,  every premorphism $\rho: S\to  T$ is strong. 

The next fact gives an important for us example of a strong premorphism.

\begin{lem}\label{tau}
Let $L$ be a Lie algebra. Then, the mapping $\tau :L \rightarrow E(L)$, given by $a\mapsto (\mathbb{F}a,a)$, is a  strong premorphism. 
  \end{lem}
  \begin{proof} It is clear that   $\tau(0)=(0,0)$. Let $a,b\in L$. Then
$$\tau (a)+\tau(b)= (\mathbb{F}a+ \mathbb{F}b,a+b)\preceq (\mathbb{F}(a+b),a+b)=\tau(a+b).$$
Moreover,
$$[\tau (a),\tau(b)] = [(\mathbb{F}a,a), (\mathbb{F}b,b)]=(\mathbb{F}a + \mathbb{F}b +\mathbb{F}[a,b],[a,b])\preceq (\mathbb{F}[a,b],[a,b])=\tau([a,b]). $$ 
 and  
$$[\tau (a),\tau(b)] = 
(\mathbb{F}a , 0) + (\mathbb{F}b, 0)  + 
(\mathbb{F}[a,b],[a,b]) = 0_{\tau (a)} + 0_{\tau (b) }
+ \tau([a,b]) = 0_{\tau (a) + \tau (b) } + \tau([a,b]).$$ 
Finally, for $\alpha\neq 0$ we have that
$\tau (\alpha a)= (\mathbb{F}(\alpha a),\alpha a )=(\mathbb{F} a,\alpha a )= \alpha (\mathbb{F} a, a )= \alpha \tau (a)$.
\end{proof}
 
 In analogy with the case of groups it is natural to give the following:
 
\begin{defn}\label{def:partialAc} 
Let $L$ be a Lie algebra and $A$ be a non-associative algebra. A collection 
$$\theta=\{\theta_x: D_{x} \rightarrow A\mid D_x\unlhd A, x\in L \}\subseteq\pde(A)$$ of partial derivations of $A$ is called {\em a partial action} of $L$ on $A,$   if the mapping 
$$\theta : L \to  \pde(A), \;\;\;  x \mapsto \theta _x, \; \; \;  (x\in L), $$    is
a premorphism.
\end{defn}

The next fact gives an alternative definition of a partial action of a Lie algebra. 

\begin{prop}\label{prop:partialAc}  
Let $L$ be a Lie algebra and $A$ be a non-associative algebra. A collection 
$$\theta=\{\theta_x: D_{x} \rightarrow A\mid D_x\unlhd A, x\in L \}\subseteq\pde(A)$$ of partial derivations of $A$ is   a partial action of $L$ on $A$ if and only if  the following properties are  satisfied:
\begin{enumerate}
     \item $D_0 = A,$ and $\theta_0$ is the zero derivation in $A.$
     \item $ D_{x}\cap D_{y} \subseteq D_{x+y}$ and $\theta_{x+y}(a)=\theta_x(a)+ \theta_{y}(a)$ for all $x,y\in L$ and $a\in  D_{x}\cap D_{y}$.
     \item $\theta_x^{-1} (D_{y} )\cap \theta_{y}^{-1} (D_x) \subseteq D_{[x,y]} $ for all $x,y\in L$ and $$\theta_{[x,y]}=\theta_x\theta_{y}(a)-\theta_y\theta_{x}(a)\mbox{ for all } a\in \theta_x^{-1} (D_{y} )\cap \theta_{y}^{-1} (D_x  ).$$
     \item $ D_{\alpha x}=D_{ x}$ and $\theta_{\alpha x}(a)=\alpha \theta_x(a)$ for all $x\in L, a\in D_x$ and $ 0\neq \alpha\in \mathbb{F}.$  
\end{enumerate} 
\end{prop}

Note that if every $D_x$ is equal to $A$, then this concept coincides with the traditional notion of an action (by derivations) of $L$ on $A$, and we shall call it a {\em global action}.

The inclusion $D_x\cap D_y\subseteq D_{x+y}$ in  Proposition~\ref{prop:partialAc} can be replaced by the equality $$D_x\cap D_y=D_{x+y}\cap D_x, \;\;\; \forall x,y \in L.$$
 Indeed, this inclusion together with condition~(4) imply $D_{x+y}\cap D_{x}=D_{x+y}\cap D_{-x}\subseteq D_y$, which yields the desired equality.

A partial action $\theta=\{\theta_x: D_{x} \rightarrow A\mid D_x\unlhd A, x\in L \}$ of $L$ on $A$ will be called {\em strong }  if  the premorphism 
$\theta : L \to  \pde(A)$ is strong. Equivalently, 
$\theta $ is strong if 
$$\theta_x^{-1} (D_{y} )\cap \theta_{y}^{-1} (D_x)=D_x\cap D_y \cap D_{[x,y]}.$$
for all $x,y\in L.$ This implies that
$$ \theta^{-1}_x(D_y)\cap D_y  \cap D_{[x,y]} =
D_x \cap \theta^{-1}_y (D_x)  \cap D_{[x,y]} = D_x\cap D_y \cap D_{[x,y]}, $$ for all $x, y\in L$. 
Indeed,  
$$D_x\cap D_y  \cap D_{[x,y]}=\theta_x^{-1} (D_{y} )\cap \theta_{y}^{-1} (D_x) \cap D_{[x,y]} \subseteq \theta^{-1}_x(D_y)\cap D_y \cap D_{[x,y]} \subseteq D_x\cap D_y \cap D_{[x,y]},$$
yielding $\theta^{-1}_x(D_y)\cap D_y \cap D_{[x,y]} = D_x\cap D_y \cap D_{[x,y]},$ and the other equality is obtained analogously.

\begin{defn}
Assume that $\mathcal A$  is a class of nonassociative algebras. Let $L$ be a Lie algebra and $A\in \mathcal{A}$. We say that a  partial action  $\theta=\{\theta_x: D_{x} \rightarrow A\mid D_x\unlhd A, x\in L \}$ of $L$ on $A$ is an {\em $\mathcal{A}$-partial action} if  each $D_x\in \mathcal{A}$.
\end{defn}

Suppose that $B$ is a nonassociative algebra and that $\eta: L\to \der(B)$ is a Lie algebra homomorphism. If $A$ is an ideal of $B$, then for each $x\in L$ we set  $D_x=A\cap \eta_x^{-1}(A)$ and $\theta_x=\eta_x|_{{D_x}}.$ It is straightforward to verify that the collection $\theta=\{\theta_x: D_{x} \rightarrow A\mid x\in L \}$ is a partial action of $L$ on $A$  such that 
$$\theta_x^{-1} (D_{y} )\cap \theta_{y}^{-1} (D_x)= D_{[x,y]}\cap \theta_x^{-1} (D_{y} )\cap D_y =
D_{[x,y]} \cap D_x \cap \theta_y^{-1} (D_{x} ) , $$ for all $x,y\in L.$ This means that the premorphism 
$\theta: L\to \pde(A), x \mapsto \theta _x ,$ is such that 
$$[\theta(x),\theta(y)]=\theta([x,y])+0_{\theta(y)\theta(x)}+0_{\theta(y)} =\theta([x,y])+0_{\theta(x)\theta(y)}+0_{\theta(x)},$$ for all $x,y\in L.$
This partial action is called the {\em restriction } of $\eta$ to $A$. If $A$ is invariant under the action $\eta$, then the restriction of $\eta$ to $A$ constitutes a global action on $A$. For instance,  if $L$ is a Lie algebra that contains an ideal $I$ that is not invariant under $\der(L)$ (e.g.~\cite[p. 75]{Ja}), then the resulting restriction of the action of $\der(L)$ to $I$ is not a  global partial action.

\begin{ex}
Let $L$ be a Lie algebra and $A$ be a nonassociative idempotent algebra, meaning $A^2=A$. Define $D_0=A$ and $D_x=0$ for $x\neq 0$, and  let $\theta_x$ be the zero map on $D_x$ for every $x\in L$. It is straightforward to check that  $\theta=\{\theta_x: D_{x} \rightarrow A\mid x\in L \}$ is a strong partial action of $L$ on $A$. Furthermore, $\theta$ is not a restriction of a global action of $L$. To see why, assume by contradiction that there exist a nonassociative algebra $B$ with $A\unlhd B$, and an action  $\eta: L\to \der(B)$, such that $D_x=\eta_x^{-1}(A)\cap A$ for every $x\in L$. Then $$\eta_x(A)=\eta_x(A^2)=A\eta_x(A)+\eta_x(A)A\subseteq A$$
yielding $A\subseteq \eta_x^{-1}(A)$. Consequently, $D_x=A$ for every $x\in L$, contradicting our  definition of $D_x$. Thus, $\theta$ cannot be a restriction of a global action of 
$L.$
\end{ex}


\begin{lem}\label{welldefined}
Let $L$ be a Lie algebra, $T$ be  a Lie inverse semialgebra, and $\rho: L\to T$ a premorphism. Assume that $A$ is a finite-dimensional subspace of $L$ and that $\beta$ is a finite set of generators of $A$. Then   
$$\inf \E(\rho(A)):=\inf\{0_{\rho(a)}\mid a\in A\}=\sum_{x_i\in \beta} 0_{\rho(x_i)}.$$ 
In particular, the sum $\sum_{x_i\in \beta} 0_{\rho(x_i)}$ does not depend of the finite set of generators of $A$. 
\end{lem}
\begin{proof}
Let $a\in A$. Then $a=\sum_{x_i\in \beta}\alpha_ix_{i},$ for some $\alpha_i\in \F$. As $\rho$ is a premorphism, we have $\rho(a)=\rho (\sum_{x_i\in \beta}\alpha_ix_{i}) 
\succeq \sum_{x_i\in \beta}\alpha_i\rho(x_{i}),$ which implies
$$0_{\rho(a)}\succeq\sum_{x_i\in \beta}\alpha_i0_{\rho(x_i)}=\sum_{x_i\in \beta}0_{\rho(x_i)}.$$
Thus $\sum_{x_i\in \beta}0_{\rho(x_i)}$ is a lower bound of $\E(\rho(A))$. Now, suppose that $\lambda\in \E(\rho(L))$ is another lower bound of $\E(\rho(A)).$ Then, in particular, $\lambda\preceq 0_{\rho(x_i)}$ for all $x_i\in \beta$, and hence $\lambda\preceq \sum_{x_i\in \beta}0_{\rho(x_i)}.$ Therefore, $\sum_{x_i\in \beta}0_{\rho(x_i)}$ is the infimum of $\E(\rho(A))$.
\end{proof}

We derive the following result as a direct  consequence of the Lemma~\ref{welldefined}.

\begin{cor}\label{cor:infErho}
Let $L,\rho$ and $T$ as defined as in Lemma~\ref{welldefined}. Then, for any finite-dimensional subspaces  $A,B$  of $L$ we have 
$$\inf\E(\rho(A+B))=\inf\E(\rho(A))+\inf\E(\rho(B)).$$
\end{cor}

The following result was inspired by a well-known result due to R. Exel~\cite[Proposition 2.2]{Exel1998} (see also~\cite[Theorem 2.4]{KL}), which establishes  a correspondence between premorphisms from a group $G$ into a inverse monoid and homomorphism from the Exel semigroup into the same inverse monoid. 

\begin{thm}\label{prouniq}
Let $L$ be a Lie algebra and let $\tau:L \rightarrow E(L)$ be as in Lemma~\ref{tau}. Let, furthermore, $S$ be a Lie inverse semialgebra with zero satisfying \eqref{eq:Sminus}. In particular, we may take $S=T^{-}$, where $T$ a right distributive associative inverse semialgebra with zero.  Suppose that $\rho :L \rightarrow S$ is a premorphism. Then,  the mapping $\widetilde{\rho}:E(L) \rightarrow S$ defined by $(A,a)\mapsto \inf\E(\rho(A))+\rho (a),$ is the unique linear map of inverse semivector spaces such that
\begin{equation}\label{eq:unique}
\widetilde{\rho} \circ \tau =\rho .
\end{equation} Moreover,  
\begin{equation}\label{eq:moreover}
[\widetilde{\rho}(A,a),\widetilde{\rho}(B,b)]=\widetilde{\rho} ([(A,a),(B,b)]) +[\rho (a),\inf\E(\rho(B)) ]+ [\inf\E(\rho(A)), \rho(b)],
\end{equation}
 for all $(A,a),(B,b)\in E(L).$ 
\end{thm}
\begin{proof}
First observe that by Lemma~\ref{welldefined}, the mapping $\widetilde \rho$ is well defined. Now, let $(A,a),(B,b)\in E(L)$ and  $0\neq \alpha\in \F$. Then  $\widetilde\rho \circ \tau(a)=\widetilde\rho(\mathbb{F}a,a)= 
 \inf(\E(\rho(\F a)))
+\rho(a)=0_{\rho(a)}+\rho(a)=\rho(a),$ which shows \eqref{eq:unique}. 

We shall check next that $\widetilde\rho$ is linear.  As $\rho$ is a premorphism, $\rho(a)+\rho(b)=\rho(a+b)+0_{\rho(b)}$, which implies 
$$\inf\E(\rho(B))+\rho(a+b)=\inf\E(\rho(B))+\rho(a)+\rho(b),$$
and thus, using Corollary~\ref{cor:infErho}, 
\begin{align*}
\tilde{\rho} ((A,a)+(B,b))&=\tilde{\rho} (A+B,a+b)=\inf\E(\rho(A+B))+\rho(a+b)\\ &=\inf\E(\rho(A))+\inf\E(\rho(B))+\rho(a+b)\\   &=\inf\E(\rho(A))+\inf\E(\rho(B))+\rho(a)+\rho(b)\\
    &= \widetilde{\rho} (A,a)+ \widetilde{\rho}(B,b).
\end{align*} Furthermore,   
$$\widetilde\rho((A,\alpha a)=\inf\E(\rho(A))+\rho(\alpha a)=\alpha\inf\E(\rho(A))+\alpha\rho(a)=\alpha(\inf\E(\rho(A))+\rho(a))=\alpha(\widetilde\rho(A,a)),$$  as  $\alpha \neq 0.$ It follows by Remark~\ref{rem:linear} that  $\widetilde\rho $ is linear.

 In order to prove \eqref{eq:moreover} note,   keeping in mind  (4) and (6) of Definition~\ref{def6}, that
\begin{align*}
[\widetilde\rho (A,a),\widetilde\rho (B,b)]&=
   [ \inf\E(\rho(A)) + \rho (a), \inf\E(\rho(B)) + \rho (b)]\\ 
     &=\inf\E(\rho(A))+\inf\E(\rho(B)) + [\rho (a),\inf\E(\rho(B)) ]+ [\inf\E(\rho(A)), \rho(b)]+ [\rho(a),\rho(b)]=(*).
\end{align*}
Now, since $\rho$ is a premorphism  and $S$ satisfies~\eqref{eq:Sminus}, we have that  $$[\rho(a),\rho(b)]=[\rho(a),0_{\rho(b)}]+ [0_{\rho (a)}, \rho (y)]+\rho([a,b]).$$ 
Additionally,  by \eqref{eq:ineq} the inequalities $\inf \E(\rho(A))\preceq 0_{\rho(a)}$  and $\inf \E(\rho(B))\preceq 0_{\rho(b)}$ imply that 
$$[\inf \E(\rho(A)),\rho(b)]\preceq [0_{\rho(a)},\rho(b)] \quad \mbox{ and }\quad 
[\rho(a), \inf \E(\rho(B))]\preceq [\rho(a), 0_{\rho(b)}].$$ 
Therefore,   using again Corollary~\ref{cor:infErho}, 
\begin{align*}
    (*)&=\inf\E(\rho(A+B)) + [\rho (a),\inf\E(\rho(B)) ]+ [\inf\E(\rho(A)), \rho(b)]+[\rho(a),0_{\rho(b)}]+ [0_{\rho (a)}, \rho (y)]+\rho([a,b])\\
    &= \inf\E(\rho(A+B)) + [\rho (a),\inf\E(\rho(B)) ]+ [\inf\E(\rho(A)), \rho(b)]+\rho([a,b])\\
    &= \inf\E(\rho(A+B)) + 0_{\rho([a,b])} + \rho([a,b]) + [\rho (a),\inf\E(\rho(B)) ]+ [\inf\E(\rho(A)), \rho(b)]\\
     &= \inf\E(\rho(A+B + \F [a,b])) + \rho([a,b]) + [\rho (a),\inf\E(\rho(B)) ]+ [\inf\E(\rho(A)), \rho(b)]\\
    &= \widetilde\rho ([(A,a),(B,b)]) +[\rho (a),\inf\E(\rho(B)) ]+ [\inf\E(\rho(A)), \rho(b)],
\end{align*} 
 completing the proof of \eqref{eq:moreover}.

Now, let us address the assertion of uniqueness. For this suppose that  $\eta:E(L) \rightarrow S$ is a   linear mapping such that $\eta \circ \tau = \rho.$  In particular, we have that 
$\eta (0_{(A,a)})= \eta (0 (A,a))= 0 \eta  (A,a)= 0_{\eta(A,a)}$ for all $(A,a)\in E(L)$.  Since 
 $$\eta (A,a)=\eta(A,0)+\eta(\mathbb F a, a)=
\eta(A,0)+\rho(a)= \eta(A,0)+\widetilde\rho(\mathbb F a, a),$$ it is enough to verify that  $\eta (A,0)=\tilde{\rho}(A,0).$  Let $\beta$ be a finite set of generators of $A.$  Then, since   $\rho (0)=0,$ we obtain  
\begin{multline*}
 \eta (A,0)= \eta (\sum_{x\in \beta} \mathbb{F} x,0)=\sum_{x\in \beta} \eta (\mathbb{F}x,0)=
     \sum_{x\in \beta} \eta (0_{(\mathbb{F}x,x)})=\\
     =\sum_{x\in \beta} 0_{\eta (\mathbb{F}x,x)}=\sum_{x\in \beta} 0_{\rho(x)}=\inf\E(\rho(A))=\tilde\rho(A,0),  
\end{multline*}
as required. 
\end{proof}
\begin{cor}\label{cor:tilderho}
Suppose that $L$ is a Lie algebra, $S$  is a semilattice of Lie algebras,  $\rho :L \rightarrow S$  is a premorphism and   $\tau:L \rightarrow E(L)$  is  as defined in Theorem~\ref{prouniq}.  Then $\widetilde{\rho}:E(L) \rightarrow S$ is a homomorphism of Lie inverse semialgebras, and it is the unique homomorphism satisfying $\widetilde\rho \circ \tau=\rho$. 
\end{cor}
\begin{proof}
By Theorem~\ref{prouniq}, it remains only to show that $\widetilde\rho$ verifies  $[\tilde{\rho}(A,a),\tilde{\rho}(B,b)]=\tilde{\rho} ([(A,a),(B,b)])$. Since  $S$ is a semilattice of Lie algebras, it follows by Lemma~\ref{lemma2_SemilatticeAlg}  that 
$$[\rho (a),\inf\E(\rho(B)) ]+ [\inf\E(\rho(A)), \rho(b)]
= 0_{\rho (a)}+\inf\E(\rho(B))+ \inf\E(\rho(A))+0_{\rho(b)}=\inf\E(\rho(A))+\inf\E(\rho(B)). $$
Then  \eqref{eq:moreover} and   Corollary~\ref{cor:infErho}, imply
\begin{align*}
[\tilde{\rho}((A,a)),\tilde{\rho}((B,b))]&=\widetilde{\rho} ([(A,a),(B,b)]) +[\rho (a),\inf\E(\rho(B)) ]+ [\inf\E(\rho(A)), \rho(b)]\\&= \inf\E(\rho(A+B)) +\rho([a,b])+ \inf\E(\rho(A))+\inf\E(\rho(B)) 
\\&=\inf\E(\rho(A+B))+0_{\rho([a,b])} +\rho([a,b])\\
& =\inf\E(\rho(A+B+\F[a,b])) +\rho([a,b]) =\tilde{\rho} ([(A,a),(B,b)]),
\end{align*}
as required.
\end{proof}

\begin{cor}\label{cor:classes}
Assume that $\mathcal A$ is any of the following classes of nonassociative algebras: unital associative algebras, unital Jordan algebras, sympathetic Lie algebras, or semisimple Lie algebras.  Let $A\in \mathcal A$  and $L$ be a Lie algebra. Then, there exists a one-to-one correspondence between:
\begin{enumerate}
    \item The $\mathcal{A}$-partial actions of $L$ on $A$, and 
    \item the homomorphism of Lie inverse semialgebras  $ E(L)\to \pde_{\mathcal{A}}(A).$
\end{enumerate}
\end{cor}
 \begin{proof} By definition, the $\mathcal{A}$-partial actions of $L$ on $A$ are  exactly the premorphisms of the form $\rho : E(L)\to \pde_{\mathcal{A}}(A).$ By Proposition~\ref{u1}, $\pde_{\mathcal{A}}(A)$ is a semilattice of Lie algebras, so that given such a $\rho, $  Corollary~\ref{cor:tilderho} gives us the unique 
homomorphism of Lie inverse semialgebras  $\widetilde\rho : E(L)\to \pde_{\mathcal{A}}(A)$ such that
 $\widetilde\rho \circ \tau=\rho$. On the other hand, for any
 homomorphism of Lie inverse semialgebras   
 $\varphi : E(L)\to \pde_{\mathcal{A}}(A)$ the composition $\rho = \varphi \circ \tau $ is obviously a premorphism. Hence  by the uniqueness property $\varphi = \widetilde{\rho} $ and we obtain the desired one-to-one correspondence.
\end{proof}

\begin{ex}
We will see that, in general, the map 
 $\widetilde{\rho}:E(L)\to S$ in Theorem~\ref{prouniq}, is not a homomorphism of Lie inverse semialgebras.
To see this, we refer to the well-known example by Jacobson involving a Lie algebra with an ideal that is not invariant under derivations, see~\cite[p. 75]{Ja}. Let us consider a field $\F$ with characteristic  $p$ greater than 2, and let $A$ be the commutative associative algebra  with basis  $\{1,z\ldots,z^{p-1}\}$, where $z^p=0$. Now, let $S$ be any simple Lie algebra, and put $L=S\otimes A$, where the bracket in $L$ is given by $[s\otimes a,t\otimes b]=[s,t]\otimes ab.$ Consider the ideal $I=S\otimes
 \text{span}\{z,\ldots,z^{p-1}\}$ of $L$ and the derivation $d:L\to L$ given by $d(s\otimes z^i)= s \otimes iz^{i-1}$, for $i=1,\ldots,p-1$. Then, we have that 
$$I_d:=I\cap  d^{-1} (I) =\{s\otimes a\in I\mid d(s\otimes a)\in I\}=S\otimes  \text{span}\{z^2,\ldots,z^{p-1}\}.$$
The restriction $\rho=\{\rho_h: I_h\to I\mid h\in \der(L)\}$ of the action of $\der(L)$ on $L$ to $I$ is a non-global partial action of $\der(L)$ on $I$. Consequently, the mapping $\rho: \der(L)\to \pde(I)$, defined by $h\mapsto \rho_h$, is a premorphism. We claim that the induced mapping   $\widetilde\rho: E(\der(L))\to \pde(I),$ given by $ \widetilde\rho(A,a)=\inf\E(\rho(A))+\rho(a),$  is not a homomorphism. Indeed, if $\widetilde\rho$ were a homomorphism, then 
$$[d|_{I_d},d|_{I_d}]=[\rho_d,\rho_d]=[\widetilde\rho(\F d,d),\widetilde\rho(\F d,d)]=\widetilde\rho(\F d,0) =0_{\rho_d},$$
which implies that the domains of the maps $[d|_{I_d},d|_{I_d}]$ and $d|_{I_d}$ must be equal, i.e. $d^{-1}(I_d)\cap I_d=I_d$. However, since  $d(s\otimes z^2)= 2s\otimes  z\notin I_d$ for $s\neq 0$, we find $s\otimes z^2\in I_d\backslash d^{-1}(I_d).$ Therefore, $\widetilde\rho:E(\der(L))\to \pde(I)$ is not a homomorphism of Lie inverse semialgebras.
\end{ex}

\section{  $F$-inverse Lie semialgebras and partial representations}\label{s: F-inverse}

\subsection{Recalling $F$-inverse semigroups}

An interesting class of inverse semigroups close to groups is formed by $F$-inverse semigroups, i.e.   inverse semigroups each  $\sigma $-class of which contains a greatest element. An important fact due to M. Szendrei states that there is an adjunction between the category of groups and the category of $F$-inverse semigroups, with those their morphisms which preserve greatest elements  of  $\sigma $-classes  \cite{Szendrei}. It is also known that the latter category is equivalent with the category of the so-called $F$-pairs and their morphisms (see \cite[VII.6]{Petrich1984}).   It is easy to see  that an $F$-pair can be identified with what we call a unital partial action of a group on a meet semilattice. For the reader's convenience we recall  the  following: 

\begin{defn}\label{def:ParGrAcOnSet}(see \cite{Exel1998} and \cite{ExelBook})   Let $G$ be a group with identity element $1_G$ and $\X$ be a set.  A partial action ${\0}$ of
$G$ on  $\X$ is a collection of  subsets ${\D}_g \subseteq \X \; (g \in G)$ and bijections
$ {\0}_g : {\D}_{g\m} \to  {\D}_g$ such that for all $g,h \in G$ the following conditions are satisfied:\\

 (i) $\D_{1_G} = \X$ and ${\0}_{1_G}$ is the identity map of $\X$;

 (ii) $\D_{(gh)\m} \supseteq {\0}\m_h({\D}_h \cap {\D}_{g\m})$;

 (iii) ${\0}_g \circ {\0}_h (x) = {\0}_{gh}(x)$ for each $x \in {\0}\m_h ({\D}_h \cap {\D}_{g\m})$.\\
\end{defn}

Note that conditions (ii) and (iii) mean that the function ${\0}_{gh}$ is an extension of the function ${\0}_g \circ
{\0}_h $. Moreover, $\0\m _g = \0_{g\m}$ and it is easily seen that  (ii) can be replaced by a ``stronger looking'' condition: 
\begin{equation}\label{eq:StrongerLooking}
{\0}\m_h({\D}_h \cap
{\D}_{g\m}) = \D_{h\m} \cap \D_{h\m g\m},
\end{equation} for all $g,h\in G$ (see \cite{ExelBook}). If $\X $ is a semigroup, then we assume that each $\X _g$ is an ideal in $\X $ and each $\0 _g$ is an isomorphism of semigroups. We say that the partial action $G$ on a semigroup $\X $ is unital if each $\D _g$ is a monoid (note that we use the term ``unital'' for partial actions in a different sense than the authors of \cite{KL}).

Let now $\0 $ be a unital partial action of a group $G$ on a meet semillatice  $\Lambda = (\Lambda , \leq).$ Denote by $e_g$ the identity element of the monoid $\D _g, $  $(g\in G).$ In particular, $e_{1_G}$ is the identity element of $\D _{1_G}= \Lambda , $  which we denote by $\varepsilon .$ Observe that \eqref{eq:StrongerLooking} can be rewritten as 
${\0}_g({\D}_{g\m} \cap
{\D}_{h}) = \D_{g} \cap \D_{gh},$ which implies that 
\begin{equation}\label{eq:IdempotentsRule}
\0_g(e_{g\m} \wedge e_h) = e_g \wedge e_{gh},
\end{equation} for all $g,h \in G.$

Let $\0=\{\0_g: \D_{g^{-1}}\to \D_g\mid g\in G, \D_g\subseteq \Lambda\}$ and $\widetilde\0=\{\widetilde\0_g:\widetilde\D_{h^{-1}}\to \widetilde\D_h\mid h\in H, \widetilde\D_h\subseteq \widetilde\Lambda\}$ be unital partial actions of the groups $G$ and $H$ on the semilattices $\Lambda$ and $\widetilde\Lambda$, respectively. A morphism between $\0$ and $\widetilde\0$ is a pair $(\theta,\phi)$, where $\psi: G\to H$ is a group homomorphism and $\phi: \Lambda\to \widetilde\Lambda$ is a monoid homomorphims, which for every $g\in G$ satisfy
$$\psi(\D_g)\subseteq \widetilde\D_{\theta(g)} \quad\mbox{ and }\quad \psi\circ \0_g=\widetilde\0_{\phi(g)}\circ \psi \mbox{ on } \D_{g^{-1}}.$$
Evidently, the collection of unital partial actions of groups on semilattices, together with their morphisms, form a category.

\begin{defn}\label{d: partial representation}
    Let \(G\) be a group and \(\Lambda\) be semilattice with unit.  A {\em unital partial representation} of \(G\) into \(\Lambda\) is a map \(\pi: G \to \en(\Lambda)\) such that for all \(g, h \in G\), the following conditions are satisfied:
    \begin{enumerate}
        \item $\im \pi_g$ is a unital ideal of $\Lambda$; 
        \item \(\pi_{1_G}\) is the identity map on \(\Lambda\); 
        \item \(\pi_{g^{-1}}\pi_{g}\pi_{h} = \pi_{g^{-1}}\pi_{gh}\);
        \item \(\pi_{g}\pi_{h}\pi_{h^{-1}} = \pi_{gh}\pi_{h^{-1}}\).
    \end{enumerate}
\end{defn}

Given a unital partial representation  \( \pi: G \to \en(\Lambda) \), for the sake of simplicity, we will denote \(\pi_g \pi_{g^{-1}}\) by \(\chi_g\), for $g\in G$. Note that  $\chi_g$ is idempotent, and which satisfies $\chi_g\pi_g=\pi_g\chi_{g^{-1}}=\pi_g$. Also, an element $\lambda\in \Lambda$ lies in $\im\pi_g$ if and only if $\chi_g(\lambda)=\lambda$.  Conditions (2) and (3) in the definition above can be rewritten as:
\[
    \pi_g \pi_h = \pi_{gh} \chi_{h^{-1}} \quad \text{and} \quad \pi_g \pi_h = \chi_g \pi_{gh}.
\]
Moreover, for all \( \lambda,\mu \in \Lambda \) and \( g, h \in G \), the following useful relation holds:
\begin{equation} \label{eq: relations of pi}
    \pi_g(\lambda) \wedge\pi_h(\mu) = \pi_g(\lambda\wedge \pi_{g^{-1}h}(\mu)).
\end{equation}
To verify this, observe that
\[
    \pi_g(\lambda \wedge\pi_{g^{-1}h}(\mu)) = \pi_g(\lambda)\wedge \pi_g(\pi_{g^{-1}h}(\mu)) = \chi_g \pi_g(\lambda)\wedge \chi_g(\pi_h(\mu)) = \chi_g(\pi_g(\lambda)\wedge \pi_h(\mu)) = \pi_g(\lambda) \wedge\pi_h(\mu).
\]

Let $\pi: G\to \en(\Lambda)$ and $\xi: H\to \en(\Pi)$ be unital partial representations of the groups $G$ and $H$ into the unital semilattices $\Lambda$ and $\Pi$, respectively. A pair $(\theta,\phi)$, where $\theta:\Lambda\to \Pi$ is a monoid homomorphism and $\phi:G\to H$ is a group homomorphism, is a homomorphism between $\pi$ and $\xi$ if for all $g\in G$ and $\lambda\in \Lambda$ we have 
$$\theta(\pi_g(\lambda))=\xi_{\phi(g)}(\theta(\lambda)).$$
Clearly, unital partial representations with their homomorphisms form a category. 
\begin{rem}\label{r: partial actions induced by partial representations}
    Partial actions and partial representations are closely related. In fact, if $\Lambda$ is a semilattice with a unit and $\pi: G \to \operatorname{End}(\Lambda)$ is a unital partial representation, we can define a unital partial action $\theta$ by setting $\D_g = \operatorname{im} \pi_g$ and $\theta_g = \pi_g|_{\D_{g^{-1}}}$.  

Conversely, if $\theta$ is a unital partial action of $G$ on $\Lambda$, then the mapping $\pi: G \to \operatorname{End}(\Lambda)$ given by  
\[
\pi_g(\lambda) = \theta_g(e_{g^{-1}} \wedge \lambda)
\]
defines a unital partial representation. Moreover, morphisms between unital partial representations naturally induce morphisms between the corresponding unital partial actions, and vice versa. Hence, this construction establishes a one-to-one correspondence between unital partial representations of a group in unital semilattices and unital partial actions on such semilattices, establishing an isomorphism of categories between them, which can be readily verified.
\end{rem}

\begin{lem}\label{l: f-pairs} Let \(G\) be a group and \((\Lambda,\leq)\) be semilattice with unit $\varepsilon$.  A  mapping \(\pi: G \to \en(\Lambda)\) is a unital partial representation if and only if  for any $g,h\in G$ and $\lambda\in \Lambda$ we have:
\begin{enumerate}
\item $\pi_{1_G}$ is the identity map on $\Lambda$; and 
\item $\pi_g\pi_h(\lambda)=\pi_g(\varepsilon)\wedge \pi_{gh}(\lambda).$
\end{enumerate}
\end{lem}
\begin{proof}
First suppose that \(\pi: G \to \en(\Lambda)\) is a unital partial representation. We need to verify conditon (2) above. From~\eqref{eq: relations of pi} we have that 
$$\pi_g (\varepsilon)\wedge\pi_{gh}(\lambda)=\pi_g(\varepsilon\wedge\pi_{h}(\lambda))=\pi_g\pi_{h}(\lambda).$$

Conversely, suppose that \(\pi: G \to \en(\Lambda)\) verifies conditions (1) and (2). Then it is easy to see that $\im\pi_g=\{\lambda\in \Lambda\mid \lambda\leq \pi_g(\varepsilon) \}$, and hence it is an ideal of $\Lambda$ with unit $\pi_g(\varepsilon).$ Now, we see that
$$\pi_{g^{-1}}\pi_{gh}(\lambda)=\chi_{g^{-1}}(\varepsilon\wedge \pi_{g^{-1}}\pi_{gh}(\lambda))=\chi_{g^{-1}}(\varepsilon)\wedge \pi_{g^{-1}}\pi_{gh}(\lambda)=\pi_{g^{-1}}(\pi_g(\varepsilon)\wedge \pi_{gh}(\lambda))=\pi_{g^{-1}}\pi_g\pi_h(\lambda),$$
showing (3) of Definition~\ref{d: partial representation}. Finally, keeping in mind equation~\eqref{eq: relations of pi} we obtain
$$\pi_{gh}\pi_{h^{-1}}(\lambda)=\pi_{gh}(\varepsilon)\wedge \pi_g(\lambda)=\pi_g(\lambda\wedge \pi_h(\varepsilon))=\pi_g(\pi_h\pi_{h^{-1}}(\lambda)),$$
which shows (4) of Definition~\ref{d: partial representation}.
\end{proof}

  Lemma~\ref{l: f-pairs} states that unital partial representations of groups in unital meet semilattices are precisely the $F$-pairs as defined by M. Petrich in~\cite[Section VII.6]{Petrich1984}.  For the reader’s convenience, we recall an important result from~\cite{Petrich1984}, which establishes an equivalence between the category of $F$-pairs (unital partial representations) and the category of $F$-inverse semigroups whose morphisms are homomorphisms which preserve the greatest elements of $\sigma$-classes. To this end, we first recall the notion of an $F$-inverse semigroup.  Given an inverse semigroup $S$,  the congruence $\sigma$ is defined by setting $s \,\sigma\, t$ if and only if there exists an idempotent $e$ such that $se = te$. An inverse semigroup $S$ is called \emph{$F$-inverse} if every $\sigma$-class $\overline{s} \in S/\sigma$ has a greatest element.  

\begin{thm}[\cite{Petrich1984}] \label{th: equivalence F-pairs}
The category of $F$-pairs (unital partial representations) is equivalent to the category of $F$-inverse semigroups and their homomorphisms which preserve the greatest elements of $\sigma$-classes.
\end{thm} 
As an immediate consequence of Theorem ~\ref{th: equivalence F-pairs} and Remark~\ref{r: partial actions induced by partial representations} we have the following: 
\begin{cor} The category of unital partial action of groups on unital semillatices is equivalent to the category of $F$-inverse semigroups and their homomorphisms which preserve the greatest elements of $\sigma$-classes.
\end{cor}
One of the main objective in the next subsections is to prove a Lie analogue of the Theorem~\ref{th: equivalence F-pairs}. To this end, we adapt the notions of unital partial repesentation and $F$-inverse semigroup to the Lie context.


\subsection{Partial representations} 
To avoid certain anomalies, we assume for the remainder of the paper that the base field $\F$ is of characteristic different from 2.

We start introducing the notion of a (unital) partial representation of a vector space $V$ in a unital semilattice  $\Lambda$. This is defined as a (unital) partial representation $\pi: V \to \en(\Lambda)$ of the additive group $(V, +)$ into $\Lambda$ together with an additional condition. More precisely we give the following:

\begin{defn}\label{defaction_of_L}
Let $V$ be a vector space and  $\Lambda = (\Lambda , \leq)$ be a meet semilattice with unit $\varepsilon$. A {\em  partial representation} of $V$ in $\Lambda$ is a mapping $V\times \Lambda\to \Lambda, (a,\lambda)\mapsto a \cdot \lambda$, which verifies the following axioms for all $a,b\in V, 0\neq\alpha\in\F$ and $\lambda,\mu\in \Lambda$:
\begin{enumerate}
    \item $0\cdot \lambda=\lambda$,
    \item \label{item 2}$a\cdot(\lambda\wedge \mu)=a\cdot\lambda \wedge a\cdot\mu$,
    \item $a\cdot(b\cdot \lambda)=a\cdot\varepsilon \wedge (a+b)\cdot\lambda$,
    \item $(\alpha a)\cdot\lambda=a\cdot \lambda$.
\end{enumerate}
\end{defn}

A partial representation of a Lie algebra $L$ in a meet semilattice $\Lambda$, with unit $\varepsilon$, is a  partial representation of the vector space $L$ in $\Lambda$. We shall denote this partial representation by $(\Lambda, L)$. 

Given two partial representations $(\Lambda,L)$ and $(\Pi,H)$, a morphism from $(\Lambda,L)$ to $(\Pi,H)$ is defined as a pair $(\theta,\phi)$, where $\theta:\Lambda\to \Pi$ is a monoid homomorphism and $\phi:L\to H$ is a Lie algebra homomorphism satisfying
$$\theta(a\cdot \lambda)=\phi(a)\cdot \theta(a)$$
for every $a\in L$ and $\lambda\in \Lambda.$ It is easy to see that the partial representations of Lie algebras in unital meet semilattices form a category, which we denote by $\mathcal{I}.$ 

The following result addresses some elementary properties of partial representations that will be used throughout this section.
\begin{lem}\label{lemma-fp}
Let $(\Lambda,L)$ be a partial representation. Then, for any $\lambda,\mu\in \Lambda$ and $a,b\in L$, the following holds:
\begin{enumerate}
 \item $a\cdot(a\cdot \lambda)=a\cdot\lambda$ . 
\item $a\cdot\lambda\wedge a\cdot\mu=a\cdot \lambda\wedge \mu=\lambda\wedge a\cdot \mu$   
   \item $a\cdot\varepsilon \wedge (a+b)\cdot\lambda =b\cdot\varepsilon \wedge (a+b)\cdot\lambda$
    \item  $a\cdot \lambda\leq \lambda$ for all $\lambda\in \Lambda$ and $a\in L$.
    \item The action of $L$ on $\Lambda$ is compatible with the order of $\Lambda$, that is, $\lambda\leq \mu$ implies  $a\cdot \lambda\leq a\cdot \mu$ for all $\lambda,\mu\in \Lambda$. In particular, $a\cdot \lambda\leq a\cdot \varepsilon$ for all $\lambda\in \Lambda$.
    \item $\lambda\leq a\cdot \varepsilon $ if and only if $a\cdot\lambda=\lambda.$
    \item If $a\in \text{span}\{x_1,\ldots,x_k\}$, then
\begin{equation}\label{eq-action}
    x_1\cdot(x_2\cdots(x_k\cdot \varepsilon )\cdots )=a\cdot \varepsilon\wedge x_1\cdot\varepsilon\wedge\cdots \wedge x_k\cdot \varepsilon=a\cdot(x_1\cdot(x_2\cdots(x_k\cdot \varepsilon  )  \cdots )).
\end{equation}
In particular, $x_1\cdot(x_2\cdots(x_k\cdot \varepsilon)\cdots)$ is invariant under any permutation of $x_1,\ldots,x_k$.
\end{enumerate}
\end{lem}

\begin{proof}
To see (1),  we have  from (3) and (4) of Definition~\ref{defaction_of_L} that
$$a\cdot(a\cdot \lambda)=a\cdot\varepsilon\wedge (2a)\cdot \lambda=a\cdot\varepsilon\wedge a\cdot \lambda=a\cdot \lambda,$$
as $a\cdot \varepsilon$ is the unit of the ideal $a\cdot\Lambda:=\{a\cdot \lambda\mid \lambda\in \Lambda\}.$  For item (2), observe that since $a \cdot\lambda \wedge \mu$ belongs to the ideal $a \cdot \Lambda$, the previous item implies  
$$
a \cdot \lambda \wedge \mu = a \cdot (a \cdot \lambda \wedge \mu) = (a \cdot a \cdot \lambda) \wedge (a \cdot \mu) = a \cdot \lambda \wedge a \cdot \mu.
$$ 
The second equality follows analogously.

For item~(3), note that since $a\cdot\varepsilon \wedge (a+b)\cdot\lambda\in (a+b)\cdot \Lambda$, we have
\begin{multline*}
a\cdot\varepsilon \wedge (a+b)\cdot\lambda =(a+b)\cdot((-a)\cdot\varepsilon \wedge (a+b)\cdot\lambda)=(a+b)\cdot( (-a)\cdot\varepsilon) \wedge (a+b)\cdot\lambda=\\
=(a+b)\cdot \varepsilon\wedge b\cdot \varepsilon\wedge (a+b)\cdot \lambda= b\cdot \varepsilon\wedge (a+b)\cdot \lambda,
\end{multline*}
as $(a+b)\varepsilon$ is the unit of $(a+b)\cdot \Lambda.$

Statements (4)-(6) follow directly from the itens (1)-(3). To see (7), suppose that $a=\alpha_1x_1+\cdots +\alpha_kx_k$ for some $\alpha_i\in \F$, then applying  repeatedly  (3) and (4) from Definition~\ref{defaction_of_L} we obtain 
$$x_1\cdot(x_2\cdots(x_k\cdot \varepsilon  )\cdots )=a\cdot \varepsilon\wedge (\alpha_1x_1)\cdot\varepsilon\wedge\cdots \wedge (\alpha_kx_k)\cdot \varepsilon=a\cdot \varepsilon\wedge x_1\cdot\varepsilon\wedge\cdots \wedge x_k\cdot \varepsilon,$$
showing the first equality. Now, as $0\in span\{a,x_1,\ldots,x_k\}$, it follows that 
$$a\cdot(x_1\cdot(x_2\cdots(x_k\cdot \varepsilon  ) \cdots ))=0\cdot\varepsilon\wedge a\cdot \varepsilon\wedge x_1\cdot\varepsilon\wedge\cdots \wedge x_k\cdot \varepsilon=a\cdot \varepsilon\wedge x_1\cdot\varepsilon\wedge\cdots \wedge x_k\cdot \varepsilon,$$
which shows the second equality.
\end{proof}

Let $(\Lambda,L)$ be a partial representation, and let $A$ be a finite-dimensional subspace of $L$. We use  $A\cdot \varepsilon$ to denote the set $\{a\cdot \varepsilon\mid a\in A\}$. For a  finite set of generators  $\beta=\{x_1,\ldots,x_k\}$ of $A$, we define  $$\beta\circ \varepsilon:=x_1\cdot(x_2\cdots(x_k\cdot \varepsilon )\cdots ).$$ Note that from equation~\eqref{eq-action} we have that $\beta\circ\varepsilon=a\cdot \varepsilon\wedge x_1\cdot\varepsilon\wedge\cdots \wedge x_k\cdot \varepsilon$ for any $a\in A$. In particular, taking $a=x_i$ for some $i$, we have $\beta\circ \varepsilon=x_1\cdot\varepsilon\wedge\cdots \wedge x_k\cdot \varepsilon$.

\begin{lem}\label{lemma:F-pair}
Let $(\Lambda,L)$ be a partial representation, $A$ be a finite-dimensional subspace of $L$, and $\beta=\{x_1,\ldots,x_k\}$ a finite set of generators of $A$. Then
$$\inf A\cdot\varepsilon=\beta\circ\varepsilon=x_1\cdot\varepsilon\wedge\cdots \wedge x_k\cdot \varepsilon.$$
In particular, $\beta\circ \varepsilon:=
x_1\cdot(x_2\cdots(x_k\cdot \varepsilon ) \cdots )$ does not depend on the choice of the  finite set of generators $\beta$ of $A$. 
\end{lem}
\begin{proof}
Let $a\in A$. Then, by equation~\eqref{eq-action}, we have  
$$\beta\circ \varepsilon=a\cdot \varepsilon\wedge x_1\cdot\varepsilon\wedge\cdots \wedge x_k\cdot \varepsilon\leq a\cdot \varepsilon,$$
which implies that $\beta\circ \varepsilon$ is a lower bound of $A\cdot \varepsilon$. Now, assume that $d\in \Lambda$ is another lower bound of $A\cdot \varepsilon$. In particular, $d\leq x_i\cdot \varepsilon$ for all $i=1,\ldots, k$; and this implies $d\leq \beta\circ \varepsilon$. Hence $\beta\circ\varepsilon$ is the greatest lower bound of $A\cdot \varepsilon$. 
\end{proof}
\begin{prop}\label{prop-inf} The following holds.
\begin{enumerate}
    \item If $A,B$ are finite-dimensional subspaces of $L$, then $$\inf(A+B)\cdot \varepsilon=\inf A\cdot \varepsilon\wedge \inf B\cdot \varepsilon.$$
    \item For any  $a\in L$ we have that $$a\cdot \varepsilon \wedge \inf A\cdot \varepsilon =\inf (A+\F a)\cdot \varepsilon=a\cdot (\inf A\cdot \varepsilon).$$
\end{enumerate}
\end{prop}
\begin{proof}
(1) Suppose  that $\beta=\{x_1,\ldots,x_k\}$ generates  $A$ and $\beta'=\{y_1,\ldots,y_l\}$ generates $B$. Then $\beta\cup \beta'$ generates $A+B$  and, using Lemma~\ref{lemma:F-pair}, we obtain 
\begin{align*}
\inf A\cdot \varepsilon\wedge \inf B\cdot \varepsilon &=x_1\cdot \varepsilon\wedge\cdots\wedge x_k\cdot \varepsilon \wedge y_1\cdot \varepsilon\wedge\cdots \wedge y_l\cdot \varepsilon= \inf (A+B)\cdot \varepsilon.
\end{align*}

(2) The first equality follows from item (1). For the second equality suppose that $\{x_1,\ldots,x_k\}$ generates $A$. Then $\{a,x_1,\ldots,x_k\}$ generates $A+\F a$, and thus, using again Lemma~\ref{lemma:F-pair}, we conclude that
$$\inf (A+\F a)\cdot \varepsilon=a\cdot(x_1\cdot(x_2\cdots(x_k\cdot \varepsilon))\cdots))=a\cdot(\inf A\cdot \varepsilon),$$
as required.
\end{proof}

\subsection{$F$-inverse Lie semialgebras}\label{s: F-inverse Lie semialgebras}
Let $S$ be a Lie inverse semialgebra. Inspired by the concept of a congruence on semigroups (see~\cite{Lawson,Petrich1984}), we define a {\em  congruence} on $S$ as an equivalence relation $\omega $ such that $(s,t)\in\omega $  implies that $(r+s,r+t), ([r,s],[r,t])$ and $(\alpha s,\alpha t)$ also belong to $\omega$ for all $r\in S$ and $\alpha\in \F$. It is straightforward to see that if $\omega $ is a congruence on $S$, then the quotient $S/\omega$ is a Lie inverse semialgebra with the natural operations, and the natural projection $\omega^\#: S\to S/\omega$ is a homomorphism of Lie inverse semialgebras. 

On $S$ we define the relation $\sigma $ by setting  $s\,\sigma\, t$ if and only if there exists $e\in \E(S)$ such that $s+e=t+e$. It is easy to see that $\sigma$ is a congruence and that $S/\sigma$ is a Lie algebra. Moreover, if $\omega $ is another congruence on $S$ such that $S/\omega $ is a Lie algebra, then $\sigma\subseteq \omega .$ We will refer to $\sigma$ as the {\em minimum Lie congruence } on $S$.  If $\phi:S\to T$ is a homomorphism between  Lie inverse semialgebras, there exists a unique Lie algebra homomorphism  $\phi^{\sigma} : S/\sigma \to T/\sigma $ such that $\sigma^{\#}\circ \phi=\phi^\sigma\circ \sigma^{\#}$. This means that
$\phi^{\sigma} (\overline{s}) = \overline{\phi (s)}, $ where
$\overline{s}$ stands for the $\sigma$-class of $s\in S$  and $\overline{\phi (s)}$ for that of $\phi (s).$
Furthermore, clearly, $(\phi\circ \psi)^\sigma=\phi^\sigma\circ \psi^\sigma.$

 Adapting the notion of an $F$-inverse semigroup to the Lie context, we say that a Lie inverse semialgebra $S$ is  {\em  $F$-inverse} if every $\sigma$-class $\overline{s}\in S/\sigma$ has a greatest element $m_{\overline{s}}$ and  for every $s,t\in S$, the following condition holds:
\begin{equation}\label{f-inverse}
    [s,t]=m_{\overline{[s,t]}}+0_{s+t}.
\end{equation}
 Observe that the fact that  $m_{\overline{s}}$  is the greatest element of the $\sigma$-class $\overline{s}$ means that  $\overline{s}=
\E (S) + m_{\overline{s}},$ i.e. $\overline{s}$ is the  principal order ideal generated by $ m_{\overline{s}}$ in  the additive inverse semigroup  $(S,+)$ considered as a partially ordered set.  Moreover,
$$m_{\overline{\alpha s}} = \alpha m_{\overline{s}} \; \;  \text{and} \; \; 
m_{\overline{s}} + m_{\overline{t}} \leq m_{\overline{s+t}},$$
for all $s,t \in S$ and $\alpha \in \F .$ Since $m_{\overline{0}}=0,$ it follows that the map 
$(S/\sigma , +) \to (S,+)$ given by 
$\overline{s}\mapsto m_{\overline{s}}$ is premorphism of additive inverse monoids. In particular, since  \eqref{eq-premorphism} is a property of a premorphism of inverse monoids, we have that
\begin{equation}\label{eq:S/SigmaPremorph}
m_{\overline{s}} + m_{\overline{t}} = 0_{m_{\overline{s}}} + m_{\overline{s+t}} = 0_{m_{\overline{t}}} + m_{\overline{s+t}},
\end{equation} for all $s,t \in S.$

\begin{ex}
A pivotal example of a Lie $F$-inverse semialgebra is the Lie inverse semialgebra $E(L)$ associated with any Lie algebra $L$ (see Section~\ref{s: preliminaries}).   Specifically, for $(A,a)$ and $(B,b)$ in $E(L)$, it can be verified that $(A,a)\sigma (B,b)$  if and only if $a=b$. Consequently, $(\F a,a)$ is the greatest element for the $\sigma$-class of  $(A,a)$, and the equality  
$$(A+B+\F [a,b], [a,b])=(\F [a,b],[a,b])+(A+B,0)$$
implies that $E(L)$ is an Lie $F$-inverse semialgebra. 
\end{ex}
Let  $(\Lambda,L)$ be a partial representation. We define  
$$F(\Lambda,L)=\{(\lambda,a)\in \Lambda\times L\mid \lambda\leq a\cdot \varepsilon\},$$  endowed with the following operations:
$$(\lambda,a)+(\mu,b)=(\lambda\wedge \mu,a+b)\quad\quad \quad \alpha(\lambda,a)=(\lambda,\alpha a) \quad\quad\quad [(\lambda,a),(\mu,b)]=([a,b]\cdot (\lambda\wedge \mu),[a,b])$$
for all $\alpha\in \F, \lambda,\mu\in \Lambda$ and $a,b\in L$. 

\begin{prop}
    $F(\Lambda,L)$ is a Lie $F$-inverse semialgebra.
\end{prop}
    
\begin{proof}
First, we will see that the above operations are well-defined. Suppose that $(\lambda,a),(\mu,b)\in F(\Lambda,L).$  Then, using (6) of Lemma~\ref{lemma-fp}, we have
$$\lambda\wedge\mu=\lambda\wedge\mu\wedge \varepsilon=a\cdot\lambda\wedge b\cdot\mu\wedge \varepsilon
 =a\cdot\lambda\wedge \mu\wedge b\cdot \varepsilon=
\lambda\wedge\mu\wedge a\cdot(b\cdot \varepsilon)=\lambda\wedge\mu\wedge a\cdot\varepsilon\wedge(a+b)\cdot \varepsilon\leq (a+b)\cdot \varepsilon,$$
which implies $(\lambda,a)+(\mu,b)\in F(\Lambda,L)$. Next, as $\lambda\wedge \mu\leq \varepsilon$, by (5) of Lemma~\ref{lemma-fp}, we have  $[a,b]\cdot (\lambda\wedge \mu)\leq [a,b]\cdot \varepsilon$, which ensures that $[(\lambda,a),(\mu,b)]\in F(\Lambda,L).$ Furthermore,  for any $\alpha\in \F$ we have  $(\alpha a)\cdot \varepsilon=a\cdot \varepsilon\geq \lambda$, showing that $\alpha(\lambda,a)\in F(\Lambda,L)$. Hence, the operations defined on $F(\Lambda,L)$ are indeed well-defined. It is also straightforward to verify that $F(\Lambda,L)$ is a commutative inverse semigroup, whose set of (additive)  idempotents is given by $\E(F(\Lambda,L))=\{(\lambda,0)\mid \lambda\in \Lambda\}$.  Note that $(\lambda , a) \preceq (\mu , b)$ if and only if $a=b$ and $\lambda \leq \mu.$

Next, we need to verify that conditions (1)–(6) of Definition~\ref{def6} are satisfied. However, since the other conditions are straightforward, we will focus only on conditions (2) and (5). To see (2) let $(\lambda,a),(\mu,b),(\nu,c)\in F(\Lambda,L)$. Then we have 
$$[(\nu,c),(\lambda\wedge \mu,a+b)]=([c,a+b]\cdot \nu\wedge \lambda\wedge\mu,[c,a+b]),$$
and 
$$[(\nu,c),(\lambda,a)]+ [(\nu,c),(\mu,b)]=([c,a]\cdot \nu\wedge\lambda\wedge [c,b]\cdot\nu \wedge \mu,[c,a+b])=(\nu\wedge[c,a]\cdot\lambda\wedge [c,b]\cdot\mu,[c,a+b]).$$
As $([c,a]+[c,b])\cdot \nu\wedge \lambda\wedge\mu\geq [c,a]\cdot \varepsilon \wedge ([c,a]+[c,b])\cdot \nu\wedge \lambda\wedge\mu=[c,a]\cdot([c,b]\cdot \nu\wedge\lambda\wedge \mu)=\nu\wedge[c,a]\cdot\lambda\wedge [c,b]\cdot\mu$, we conclude that $[(\nu,c),(\lambda,a)+(\mu,b)]\succeq [(\nu,c),(\lambda,a)]+ [(\nu,c),(\mu,b)],$  which gives (2).

To prove (5) note that using (4) of Lemma~\ref{lemma-fp}, we have $[a,[b,c]]\cdot(\lambda\wedge [b,c]\cdot \mu\wedge\nu) \leq \lambda\wedge\mu\wedge \nu$, and then
$$[(\lambda,a),[(\mu,b),(\nu,c)]]=([a,[b,c]]\cdot(\lambda\wedge [b,c]\cdot \mu\wedge\nu),[a,[b,c]])\preceq( \lambda\wedge\mu\wedge \nu,[a,[b,c]]).$$ Thus 
    $$J((\lambda,a),(\mu,b),(\nu,c))\preceq (\lambda\wedge\mu\wedge \nu,J(a,b,c))=(\lambda\wedge\mu\wedge \nu,0)=0_{(\lambda,a)+(\mu,b)+(\nu,c)}.$$
Therefore, $F(\Lambda,L)$ is a Lie inverse semialgebra.

It remains to  prove that  $F(\Lambda,L)$ is  $F$-inverse. For note that $(\lambda,a)\sigma(\mu,b)$ if and only if $a=b$. Hence $F(\Lambda,L)/\sigma=
\{\overline{(\cdot,a)}\mid a\in L\}\cong L$. Since $\lambda\leq a\cdot\varepsilon$ for every $(\lambda,a)\in F(\Lambda,L)$, it follows that 
 $\overline{(\cdot,a)}= 
((a\cdot \varepsilon) \wedge \Lambda, a) $ and that $(a\cdot\varepsilon,a)$ is the greatest element of the $\sigma$-class $\overline{(\cdot,a)}.$ Finally, the equality $$([a,b]\cdot \varepsilon,[a,b])+(\lambda\wedge \mu,0)=([a,b]\cdot \varepsilon\wedge \lambda\wedge\mu,[a,b])=([a,b]\cdot (\lambda\wedge \mu),[a,b])=[(\lambda,a),(\mu,b)]$$
shows that the equation~\eqref{f-inverse} holds for $F(\Lambda,L)$. So, $F(\Lambda,L)$ is indeed an Lie $F$-inverse semialgebra.
\end{proof}

\begin{ex}\label{ex-section4}
Let $L$ be a Lie algebra, and denote by $P_f(L)$ the set of all finite-dimensional subspaces of $L$. Then $(P_f,+)$ is a meet semilattice whose unit is the subspace $0$. It is straightforward to see that the map  $L\times P_f(L) \to P_f(L)$ defined by $a\cdot A=A+\F a$ gives a partial representation $(P_f(L),L)$. A simple calculation shows that $F(P_f(L),L)$ coincides with the Lie $F$-inverse semialgebra $E(L)$.
\end{ex}

In what follows,  $\mathcal{F}$ denotes the category whose objects are the Lie $F$-inverse semialgebras, and whose morphism are the homomorphisms which map the greatest element of a $\sigma$-class to the greatest element of a $\sigma$-class.

Suppose that $(\theta,\phi):(\Lambda,L)\to (\Pi,H)$ is a morphism of partial representations. Then, it is straightforward to verify that the mapping $F(\theta,\phi): F(\Lambda,L)\to F(\Pi,H)$, defined by $(\lambda,a)\mapsto (\theta(\lambda),\phi(a))$, is a homomorphism of Lie $F$-inverse semialgebras. Furthermore, for any $a\in L$ we have  
$$F(\theta,\phi)(a\cdot\varepsilon,a)= (\theta(a\cdot \varepsilon ),\phi(a))=(\phi(a)\cdot \varepsilon,\phi(a)),$$ which shows that $F(\theta,\phi)$ sends the greatest element of a $\sigma$-class to the greatest element of a corresponding $\sigma$-class.
 Thus we obtain a functor $F:\mathcal{I} \to \mathcal{F}$ sending $(\Lambda,L)$ to $F(\Lambda,L)$ and $(\theta,\phi)$ to $F(\theta,\phi)$. Our  objective is to show that the category $\mathcal{I}$, of partial representations, is equivalent to the category $\mathcal{F}$. To achieve this, we need to construct a functor in the opposite direction. 

Let $S$ be a Lie  $F$-inverse semialgebra, and let $m_{\overline{s}}$ denote the greatest element of the $\sigma$-class $\overline{s}$. Consider the Lie algebra $S/\sigma$, and the meet semilattice $\E(S)$ of (additive) idempotents of $S$. Then, it is easy to see that a mapping $S/\sigma\times \E(S)\to \E(S)$ defined by 
$$\overline{s}\cdot \lambda=0_{m_{\overline{s}}}+\lambda,$$
 gives a partial representation $(\E(S), S/\sigma)$. Indeed, conditions (1), (2) and (4) of 
 Definition~\ref{defaction_of_L} are straightforward, whereas (3) is a consequence of 
 \eqref{eq:S/SigmaPremorph}. Moreover, if $\varphi: S\to T$ is a homomorphism of Lie $F$-inverse semialgebras  which maps the greatest element of a $\sigma$-class to the greatest element of a $\sigma$-class, then the pair $(\theta,\varphi ^{\sigma})$,  where $\theta=\varphi|_{\E(S)}$ and $\varphi ^{\sigma}: S/\sigma\to T/\sigma $ is as defined above, 
  is a morphism in $\mathcal{I}.$ Thus, we obtain a functor $K:\mathcal{F}\to \mathcal{I}$ setting $K(S)=(\E(S),S/\sigma)$ and $ K(\varphi)= (\theta,\varphi ^{\sigma}).$

For every $(\Lambda,L)\in \mathcal{I}$ we define the mapping $(\xi_\Lambda,\eta_L): (\Lambda,L)\to KF(\Lambda,L)$, where $\xi_\Lambda(\lambda)=(\lambda,0)$ and $\eta_L(a)=\overline{(\cdot,a)}$. On the opposite direction, for $S\in \mathcal{F}$ we define the mapping $\gamma_S: S\to FK(S)$ by $\gamma_S(s)=(0_s,\overline{s}).$ The following theorem shows that the  mappings $(\xi_\Lambda,\eta_L): (\Lambda,L)\to KF(\Lambda,L)$ and $\gamma_S: S\to FK(S)$ establish an equivalence between the categories $\mathcal{I}$ and $\mathcal{F}$. The proof is essentially the same as~\cite[Lemma VII.6.7]{Petrich1984} and \cite[Lemma VII.6.8]{Petrich1984}, with some minor adaptations. However, to keep our exposition self-contained, we provide the details of the proof. 

\begin{thm}\label{pro-equivalence}
The mapping $(\xi_\Lambda,\eta_L)$ establishes an equivalence between the functors $id_{\mathcal{I}}$ and $KF$, while the mapping $\gamma_S$ establishes an equivalence between the functors $id_{\mathcal{F}}$ and $FK$.  Consequently, the categories $\mathcal{I}$ and $\mathcal{F}$ are equivalents.
\end{thm}

\begin{proof}
First note that for any $(\Lambda,L)\in \mathcal{I}$ we have that $\xi_\Lambda$ is an isomorphism between $\Lambda$ and $\E(F(\Lambda,L))$, and $\eta_{L}$ is an isomorphism between $L$ and  $F(\Lambda,L)/ \sigma$.  Furthermore, for any $a\in L$ and $\lambda\in \Lambda$ we have $$\xi_\Lambda(a\cdot \lambda)=(a\cdot \lambda,0)=(\varepsilon\wedge a\cdot \lambda,0)=(a\cdot \varepsilon\wedge \lambda,0)=(a\cdot \varepsilon,0)+(\lambda,0)=0_{m_{\overline{(\cdot,a)}}}+(\lambda,0)=
\eta_L(a)\cdot  \xi_\Lambda(\lambda ),$$
which shows that $(\xi_\Lambda,\eta_L)$ is an isomorphism in the category $\mathcal{I}$. In order to  show that the correspondence $(\Lambda,L)\mapsto(\xi_\Lambda,\eta_L)$ is natural let  $(\theta,\phi):(\Lambda,L)\to (\Pi,H)$ be a morphism of partial representations, and for convenience, set $(\theta',\phi')=KF(\theta,\phi).$ Then, for every $\lambda\in \Lambda$ and $a\in L$, we have 
$$(\theta',\phi')\circ(\xi_\Lambda,\eta_L)(\lambda,a)=(\theta',\phi')((\lambda,0),\overline{(\cdot,a)})=((\theta(\lambda),0),\overline{(\cdot, \phi(a))}=(\xi_\Pi,\eta_H)\circ(\theta,\phi)(\lambda,a),$$
showing that the correspondence $(\Lambda,L)\mapsto(\xi_\Lambda,\eta_L)$ is a natural equivalence. 

We now address the second assertion. First note that for any  $S\in \mathcal{F}$ the mapping $\gamma_S$ is a bijection. Indeed, if  $s,t\in S$ are such that $0_s=0_t$ and $\overline{s}=\overline{t}$, then $s=m_{\overline{s}}+0_s=m_{\overline{t}}+0_t=t$, showing that $\gamma_S$ in injective. Now, consider any $(\mu,\overline{s})\in FK(S)$. Since $\mu\leq 0_{m_{\overline{s}}}$, setting $t=\mu+m_{\overline{s}}$ we obtain 
$$\gamma_S(t)=(\mu+0_{m_{\overline{s}}}, \overline{(\mu+m_{\overline{s}}}))=(\mu,\overline{m_{\overline{s}}})=(\mu,\overline{s}),$$
which shows that $\gamma_S$ is onto. Furthermore, it is clear that  for any $\alpha\in \F$ and $s,t\in S$ we have $\gamma_S(s+t)=\gamma_S(s)+\gamma_S(t)$  and $\gamma_S(\alpha s)=\alpha\gamma_S(s)$.  In addition,   using \eqref{f-inverse}, we see that 
$$[\gamma_S(s),\gamma_S(t)]=[(0_s,\overline{s}),(0_t,\overline{t})]=(\overline{[s,t]}\cdot 0_{s+t},\overline{[s,t]})=(0_{m_{\overline{[s,t]}}}+0_{s+t},\overline{[s,t]})=(0_{[s,t]},\overline{[s,t]})=\gamma_S([s,t]),$$
which implies that $\gamma_S$ is an isomorphism in the category $\mathcal{F}.$ It remains to show that the correspondence $S\mapsto \gamma_S$ is natural. For this suppose that $\varphi: S\to T$ is a homomorphism in $\mathcal{F}$. Then, for any $s\in S$ we have that 
$$FK(\varphi)\circ\gamma_S(s)=FK(\varphi)(0_s,\overline{s})=(\varphi(0_s),\overline{\varphi(s)})=(0_{\varphi(s)},\overline{\varphi(s)})=\gamma_T\circ\varphi(s),$$
which gives the naturality of the correspondence $S\mapsto \gamma_S.$
\end{proof}

\subsection{Lie analogue of Szendrei theorem }

Suppose that $\phi: L\to H$ is a  Lie algebra homomorphism. Then, $\phi$ induces a homomorphism $\widetilde \phi: E(L)\to E(H)$ given by $\widetilde\phi(A,a)=(\phi(A), \phi(a))$. This homomorphism maps the greatest elements of the $\sigma$-class of $(A,a)$ to the greatest element of the $\sigma$-class of $\widetilde\phi(A,a)$. Additionally, it is important to note that $\widetilde{\phi\circ\psi}=\widetilde\phi \circ\widetilde\psi$.

Let $\mathcal{L}$ denote the category of Lie algebras, and let $\mathcal{F}$ denote the category of Lie $F$-inverse semialgebras and their homomorphisms which map the greatest element of a $\sigma$-class to the greatest element of a $\sigma$-class. By the previous paragraphs, the correspondence $L\mapsto E(L),$  $\phi\mapsto \widetilde \phi$ defines a functor $E:\mathcal{L}\to \mathcal{F}$. Similarly, the correspondence $S\mapsto S/\sigma ,$   $\phi\mapsto \phi^\sigma$ defines a functor $Q:\mathcal{F}\to \mathcal{L}.$ The main result of this section is the following theorem, which is  a Lie analogue  of a well-known result by M. Szendrei \cite[Theorem 2]{Szendrei}. 
\begin{thm}\label{theorem-action}
    The functor $E:\mathcal{L}\to \mathcal{F}$ is left adjoint to the functor $Q:\mathcal{F}\to \mathcal{L}.$
\end{thm}

The remainder of this section is devoted to  prove Theorem~\ref{theorem-action}. The proof follows the approach from~\cite{Szendrei}, with additional ingredients needed for our Lie context. More specifically, instead of proving the theorem directly,  we factorize the functors $E:\mathcal{L}\to \mathcal{F}$ and $Q:\mathcal{Q}\to \mathcal{L}$ in Theorem~\ref{theorem-action}  through the functors $F$ and $K$ described in Subsection~\ref{s: F-inverse Lie semialgebras}. To do this, we first recall from Example~\ref{ex-section4}, that for a given Lie algebra $L$   we have  the meet semilattice $P_f(L)$ of all finite-dimensional subspaces of $L$, and the mapping $L\times P_f(L) \to  P_f(L)$ defined by $a\cdot A=A+\F a$, gives a partial representation $(P_f(L),L)$. Furthermore, if $\phi: L\to H$ is a Lie algebra homomorphism, then the mapping $\widehat\phi:P_f(L)\to P_f(H)$, defined by  $\widehat{\phi}(A)=\phi(A)$, is a homomorphism of monoids; and, clearly, the mapping $(\widehat\phi,\phi):(P_f(L),L)\to  (P_f(H), H)$ is a morphism of partial representations. Moreover, we have that $\widehat{\phi\circ\psi}=\widehat\phi\circ \widehat\psi$.  Therefore, we can define a functor $\overline{E}: \mathcal{L}\to \mathcal{I}$ by setting $\overline{E}(L)=(P_f(L),L)$ and $\overline{E}(\phi)= (\widehat{\phi}, \phi)$. Similarly, we can define a functor $\overline{Q}:\mathcal{I}\to \mathcal{L}$ by setting  $\overline{Q}(\Lambda,L)=L$ and $\overline{Q}(\theta, \phi)=\phi$. The functors $\overline E$ and $\overline{Q}$ satisfy  $E=F\overline{E}$ and $Q=\overline{Q}K$, where $E$ and $Q$ are the functors in Theorem~\ref{theorem-action}. By virtue of Theorem~\ref{pro-equivalence}, to prove Theorem~\ref{theorem-action} it suffices to show the following theorem.

\begin{thm}
The functor $\overline{E}$ is left adjoint to the functor $\overline{Q}.$
\end{thm}

\begin{proof}
Let $L$ be a Lie algebra and $(\Pi,H)$ a partial representation. Then, for each homomorphism  
$\phi: L\to  H=\overline{Q}(\Pi , H)$, 
let $\beta(\phi):\overline{E}(L)\to (\Pi,H)$ be defined by $\phi\mapsto (\Theta_\phi,\phi)$, where $\Theta_\phi: P_f(L)\to \Pi$ is defined by $\Theta_\phi(A)=\inf \phi(A)\cdot \varepsilon$. It follows from Proposition~\ref{prop-inf}, that the pair  $(\Theta_\phi,\phi)$ is a homomorphism of partial representations. The mapping $\beta:\Hom(L,\overline{Q}(\Pi,H))\to \Hom(\overline{E}(L), (\Pi,H))$, defined by $\phi\mapsto \beta(\phi)$, is clearly injective. To see that $\beta$  is onto, note that for $(\Theta,\phi)\in \Hom(\overline{E}(L), (\Pi,H))$ and for any $a\in L$ we have that 
$$\Theta(\F a)=\Theta(a\cdot\varepsilon)=\phi(a)\cdot \varepsilon=\Theta_\phi(\F a),$$
which implies that $\Theta=\Theta_\phi$. Therefore the mapping $\beta $ is bijective. To complete the proof, we need to establish the naturality of  $\beta =\beta _{L,(\Pi,H)}$ in both components.

Naturality of $\beta$ in the first component: Let $\gamma:L\to K$ be a Lie algebra homomorphism. For $\phi\in \Hom(K,H)$, we have that 
$$\beta(\Hom(\gamma,H)(\phi))=\beta(\phi\gamma)=(\Theta_{\phi\gamma},\phi\gamma),$$
and on the other hand,
$$\Hom((\widehat\gamma,\gamma),(\Pi,H))(\beta(\phi))=\Hom((\widehat\gamma,\gamma),(\Pi,H))(\Theta_\phi,\phi)=(\Theta_\phi\widehat\gamma,\phi\gamma).$$
Now, for any $A\in P_f(L)$ we see that 
$$\Theta_\phi\hat\gamma(A)=\Theta_\phi\gamma(A)=\inf \phi\gamma(A)\cdot\varepsilon=\Theta_{\phi\gamma}(A),$$
which implies $\Theta_\phi\hat\gamma=\Theta_{\phi\gamma}.$ Therefore, $\beta$ is natural in the first component. 

Naturality of $\beta$ in the second component: Let $(\Theta,\phi):(\Pi, H) \to (\Lambda,M)$ be a homomorphisms of partial representations. Then, for $\psi\in \Hom(L,H)$, we have that 
$$\Hom(\overline{E}(L),(\Theta,\phi))(\beta(\psi))=\Hom(\overline{E}(L),(\Theta,\phi))(\Theta_\psi,\psi)=(\Theta\Theta_\psi,\phi\psi)$$
and 
$$\beta(\Hom(L,\phi)(\psi))=\beta(\phi\gamma)=(\Theta_{\phi\gamma},\phi\gamma).$$
We need to verify the equality $\Theta\Theta_\psi=\Theta_{\phi\gamma}$. For $A\in P_f(A),$ suppose that $\{x_1,\ldots,x_k\}$ is a finite set of generators for $A$. Since $(\Theta,\phi)$ is a homomorphism of partial representations, we have 
\begin{multline*}
    \Theta\Theta_\psi(A)=\Theta(\inf\psi(A)\cdot \varepsilon)=\Theta(\psi(x_1)\cdot\varepsilon\wedge\cdots\wedge \psi(x_k)\cdot\varepsilon)=\\=\phi(\psi(x_1))\cdot\varepsilon\wedge\cdots\wedge\phi(\psi(x_k))\cdot\varepsilon=\inf \phi\psi(A)\cdot\varepsilon=\Theta_{\phi\psi}(A).
\end{multline*}
Hence $ \Theta\Theta_\psi=\Theta_{\phi\psi}$, which establishes the naturality of $\beta$ in the second component.
\end{proof}





\section*{Acknowledgments}

The first named author was partially supported by 
Funda\c c\~ao de Amparo \`a Pesquisa do Estado de S\~ao Paulo (Fapesp), process n°:  2020/16594-0, and by  Conselho Nacional de Desenvolvimento Cient\'{\i}fico e Tecnol{\'o}gico (CNPq), process n°: 312683/2021-9. The second named author was supported by Fapesp, process n°: 2022/00953-7. The third named author was partially supported by Fapesp, process n°: 2019/08659-8, and PRPI da Universidade de São Paulo, process n°: 22.1.09345.01.2.


\begin{thebibliography}{99}

 \bibitem{AB3}  M.\ M.\ S.\ Alves,   E.\ Batista,  {\em Globalization theorems for partial Hopf
(co)actions and some of their applications},  Contemp. Math. {\bf  537}  (2011), 13--30.

\bibitem{ABV}  M.\ M.\ S.\ Alves,   E.\ Batista,  J.\ Vercruysse, {\em Partial representations of Hopf algebras},  J. Algebra,
{\bf  426}, 15  (2015),  137--187.

\bibitem{AraE1} P.\  Ara,  R.\  Exel, 
{\em Dynamical systems associated to separated graphs, graph algebras, and paradoxical decompositions},   Adv.  Math., 
{\bf 252} (2014), 748--804.   




\bibitem{AraL} P.\ Ara, M.\ Lolk,
Convex subshifts, {\em separated Bratteli diagrams, and ideal structure of tame separated graph algebras},
 Adv. Math., {\bf  328}, (2018),  367--435.

\bibitem{AzMaPaSi} D.\ Azevedo, G.\ Martini, A.\ Paques, L.\ Silva, {\em Hopf algebras arising from partial (co)actions},  J. Algebra Appl., {\bf 20,} No. 1, Article ID 2140006, 26 p. (2021).

\bibitem{BP} D.\ Bagio,  A.\ Paques, {\em Partial groupoid actions: globalization, Morita theory and Galois theory},  Comm. Algebra, {\bf 40}  (2012), 3658--3678.

\bibitem{Ba} E.\ Batista, {\em Partial actions: What they are and why we care},  Bull. Belg. Math. Soc. - Simon Stevin, {\bf 24}  (2017), no. 1, 
 35--71.
 


\bibitem{BaHaSaVe} E.\ Batista, W.\ Hautekiet, P.\ Saracco, 
J.\ Vercruysse, {\em Towards a classification of simple partial comodules of Hopf algebras},  J. Algebra, {\bf 664}, (2025), 312--347.

\bibitem{BatiVerc1} E. Batista, J. Vercruysse, {\em Dual Constructions for Partial Actions of Hopf Algebras}, J. Pure Appl. Algebra, {\bf 220} (2016), (2),  518--559.

\bibitem{BR2} J.-C. Birget, J. Rhodes, {\em Group theory via global semigroup theory},
 J.  Algebra, {\bf 120}  (1989), 284--300.

\bibitem{BresarPereraOrtegaMolina} M. Bre{\u s}ar, F. Perera, J. S. Ortega, M. S. Molina, 
{\em Computing the maximal algebra of quotients of a Lie algebra}, 
Forum Math. {\bf 21} (2009), no. 4, 601--620.


\bibitem{CaenFier} S. Caenepeel, T. Fieremans, 
{\em Galois corings and groupoids acting partially on algebras},
 J. Algebra Appl., {\bf 20}, (2021),  No. 1, Article ID 2140003, 19 p.

\bibitem{CaenJan} S. Caenepeel, K. Janssen, {\em Partial (Co)Actions of
Hopf Algebras and Partial Hopf-Galois Theory},  Commun.  Algebra, {\bf 36} (2008), 2923--2946.

\bibitem{CasPaqQuaSant} F. Castro, A. Paques, G. Quadros, A. Sant'Ana,  {\em Partial actions of weak Hopf algebras: Smash product, 
globalization and Morita theory},   J. Pure Appl. Algebra, {\bf 219}   (2015), 5511--5538.

\bibitem{Chapman} A. Chapman, L. Gatto, and L. Rowen, {\em Clifford semialgebras.} Rend. Circ. Mat. Palermo, II, {\bf 72} (2023), 1197--1238 .

\bibitem{2}
Clifford, A.H. and Preston, G.B, {\em The algebraic theory of semigroups}, vol. 1. Amer. Math. Soc. Surveys, 7(1961), p.1967.

\bibitem{CornGould} C. Cornock, V. Gould, Proper restriction semigroups and partial actions, 
{\it J. Pure Appl. Algebra} {\bf 216} (2012), 935--949.



\bibitem{Cortes2023} W. Cortes, J.L.V. Rodríguez, Globalization of partial group actions on non-associative algebras, {\em J. Algebra Appl.} (2023), https://doi .org /10 .1142 /S0219498824501391, in press.

\bibitem{dokuchaev2011} M. Dokuchaev, {\em Partial actions: a survey}, Groups,  algebras and applications, XVIII Latin American algebra colloquium.  Proceedings.  Contemporary Math., {\bf 537}  (2011),  173--184.

\bibitem{dokuchaev2019}\rule{11mm}{0.1mm}, {\em Recent developments around partial actions}, São Paulo J. Math. Sci., {\bf 13} (2019), no. 1, 195--247. 

\bibitem{DE2} M.\ Dokuchaev, R.\ Exel, {\em Partial actions and subshifts},   J.~Funct.~Analysis, {\bf 272}  
(2017), 5038--5106.


\bibitem{Exel1998} R. Exel, {\em Partial actions of groups and actions of inverse semigroups},  Proc. Amer.
Math. Soc. {\bf 126} (1998), 3481--3494.

\bibitem{exel1994} R. Exel, {\em Circle actions on $C^*$-algebras, partial automorphisms, and a generalized Pimsner-Voiculescu exact sequence}, J. Funct. Anal., {\bf 122} (1994), no. 2, 361--401.




\bibitem{E0}\rule{11mm}{0.1mm}, {\em Twisted partial actions: a classification of regular
$C^*$-algebraic bundles,}  Proc.\ London Math. Soc., {\bf 74}   (1997), no. 3, 417--443.




\bibitem{ExelBook} R.\ Exel, {\it Partial Dynamical Systems, Fell Bundles and Applications}, Mathematical Surveys and Monographs, vol. 224, American Mathematical Society, 2017.

\bibitem{Gilbert} N.\ D.\ Gilbert, {\em Actions and expansions of ordered groupoids},  J. Pure Appl. Algebra\textcolor{red}{,} {\bf 198}
(2005), 175--195.

\bibitem{Gato} L. Gatto, and L. Rowen, {\em Grassmann semialgebras and the Cayley-Hamilton theorem.} Proc.  Amer. Math. Soc., Series B 7.16 (2020): 183--201.

\bibitem{Glazek} K. Glazek. {\em A guide to the literature on semirings and their applications in mathematics and information sciences: with complete bibliography.} Dordrecht: Kluwer Academic Publishers, 2002.

\bibitem{Golan} J. S. Golan. {\em Semirings and their Applications.} Springer Science \& Business Media (2013).

\bibitem{GR3}  D.\ Gon\c calves, D.\ Royer,  {\em Infinite alphabet edge shift spaces via ultragraphs and 
their $C^*$-algebras}, 
 Int. Math. Res. Notices, {\bf 2019}, Issue 7,  2019,  2177--2203. 

\bibitem{GouHol1} V.\ Gould, C.\  Hollings,  {\em Partial actions of inverse and weakly left E-ample semigroups,} 
 J.  Aust. Math. Soc., {\bf 86}  (2009), no. 3,  355--377.

\bibitem{GreenMarcos} E.\ L.\ Green, E.\ N.\  Marcos, 
{\em Graded quotients of path algebras: A local theory},
 J. Pure Appl. Algebra, {\bf 93}  (1994), no. 2, 195--226.
 
 
\bibitem{Hol1} C.\ Hollings,  {\em Partial actions of monoids,} {\it Semigroup Forum,} {\bf 75} (2007), no. 2, 293--316. 

 \bibitem{HuVerc}  J. Hu, J. Vercruysse, {\em Geometrically partial actions}, 	 Trans. Amer. Math.
Soc., {\bf 373} (2020), 4085--4143.  

\bibitem{Ja} N. Jacobson, {\em Lie algebras}, Dover Publications, Inc., New York, 1979. Republication of the 1962
original.

\bibitem{KL}
J. Kellendonk, and M. V. Lawson, {\em Partial actions of groups},  Internat. J. Algebra Comput., 14.01 (2004): 87--114.

\bibitem{Khry1} M.\ Khrypchenko, {\em Partial actions and an embedding  theorem for inverse semigroups},
 Period. Math. Hungarica, {\bf 78}, No. 1,  (2019),
47--57.


\bibitem{Kud} G.\ Kudryavtseva, {\em Partial monoid actions and a class of restriction semigroups}, 
 J. Algebra,  {\bf 429} (2015), 342--370.
 
 \bibitem{KudLaan2} G. Kudryavtseva, V. Laan, {\em Globalization of partial actions of semigroups}, 
 Semigroup Forum, 107, No. 1, 200--217 (2023),
arXiv:2206.06808.

\bibitem{KudFur} G.\ Kudryavtseva, A. L. Furlani,
{\em A new approach to universal F-inverse monoids in enriched signature},
Result. Math. {\bf 79}, No. 7, Paper No. 260, 13 p. (2024).

\bibitem{Lawson}  M. V. Lawson, {\em Inverse semigroups. The theory of partial symmetries}, World Scientific, Singapore, 1998.

\bibitem{MaPi}  V. Mar\'{\i}n, H. Pinedo, {\em Partial Groupoid Actions on Categories: Globalization and the smash 
product},  J. Algebra Appl., {\bf 19} (2020), no. 5, Article ID 2050083, 22 p.

\bibitem{Mc} K.\ McClanahan,   {\em K-theory for partial crossed products by discrete groups},
 J. Funct. Anal., {\bf 130}  (1995), (1), 77--117.

\bibitem{Megre2}  M.\ G.\ Megrelishvili, L. Schr\"oder, {\em Globalization of confluent partial actions on topological and 
metric spaces},  Topology Appl., {\bf 145}  (2004), 119--145.


\bibitem{MiSt} D. Milan, B. Steinberg,  {\em On inverse semigroup $C^*$-algebras and crossed products}, 
 Groups Geom. Dyn., {\bf 8}, No. 2, (2014),  485--512.

\bibitem{Molina2004} M. S. Molina, {\em Algebras of quotients of Lie algebras},  J. Pure Appl. Algebra, {\bf 188} (2004), no. 1, 175--188.

\bibitem{Nystedt} P.\ Nystedt, {\em Partial Category Actions on Sets and Topological Spaces},  
 Commun. Algebra, {\bf 46},  (2)   (2018), 671--683.

\bibitem{NysOinPin} P.\ Nystedt, J.\ {\"O}inert, Pinedo, H., {\em Atinian and noetherian partial skew groupoid rings},   J. Algebra,   {\bf 503}  (2018), 433--452.

\bibitem{OrtegaMolina} J. S. Ortega, M. S. Molina, 
{\em Finite gradings of Lie algebras}, 
J. Algebra, {\bf 372} (2012), 161--171.

\bibitem{Petrich1984} M. Petrich, {\em Inverse Semigroups}, Wiley, New York, 1984.

\bibitem{RenWil} J.\ N.\ Renault,  D.\ P.\ Williams, {\it Amenability of groupoids arising from partial semigroup actions and topological higher rank graphs},
  Trans.\ Am.\ Math.\ Soc.,  {\bf 369}  (2017), No.4, 2255--2283.

\bibitem{Szendrei} M. B. Szendrei. A note on Birget-Rhodes expansion of groups. {\em Journal of Pure and Applied Algebra} 58.1 (1989): 93--99.

\bibitem{Zimmermann} U. Zimmermann. {\em Linear and combinatorial optimization in ordered algebraic structures.} Elsevier, 2011.
\end{thebibliography}
\end{document}